\documentclass[12pt,oneside,reqno]{amsart}
\usepackage[latin1]{inputenc}
\usepackage[english]{babel}
\usepackage{amsmath}
\usepackage{amssymb}
\usepackage{enumitem}
\usepackage[paper=a4paper]{geometry}
\usepackage{mathrsfs} 
\usepackage{mathtools}
\usepackage{hyperref}
\usepackage{color}

\usepackage[square,numbers]{natbib}

\title{Bessel SPDEs and renormalised local times}

\author{Henri Elad Altman}
\author{Lorenzo Zambotti}
\address{Sorbonne Universit\'{e} \\
          Laboratoire de Probabilit\'{e}s Statistique et Mod\'{e}lisation\\
          4 Pl. Jussieu, 75005 Paris, France}
\email{eladaltman@lpsm.paris, zambotti@lpsm.paris}

\date{\today}

\newfont{\indic}{bbmss12}

\def\un#1{\hbox{{\indic 1}$_{#1}$}}
\newcommand{\R}{\mathbb R}
\newcommand{\N}{\mathbb N}
\newcommand{\Z}{\mathbb Z}
\newcommand{\E}{\mathbb E}
\newcommand{\gep}{\varepsilon}

\renewcommand{\d}{\, \mathrm{d}}

\theoremstyle{plain}
\newtheorem{thm}{Theorem}[section]
\newtheorem{lm}[thm]{Lemma}
\newtheorem{prop}[thm]{Proposition} 
\newtheorem{cor}[thm]{Corollary} 
  
\theoremstyle{definition}
\newtheorem{df}[thm]{Definition}

\newtheorem{rk}[thm]{Remark}
\numberwithin{equation}{section}

\begin{document}


\begin{abstract}
In this article, we prove integration by parts formulae (IbPFs) for the laws of Bessel bridges from $0$ to $0$ over the interval $[0,1]$ of dimension smaller than $3$. As an application, we construct a weak version of a stochastic PDE having the law of a one-dimensional Bessel bridge (i.e. the law of a reflected Brownian bridge) as reversible measure, the dimension 1 being particularly relevant in view of applications to scaling limits of dynamical critical pinning models.
We also exploit the IbPFs to conjecture the structure of the stochastic PDEs associated with Bessel bridges of all dimensions smaller than $3$. 
\end{abstract}

\maketitle

\section{Introduction} 

The classical stochastic calculus due to Kiyosi It\^o was originally created as a tool to 
define and solve stochastic differential equations (SDEs). In classical monographs on the subject, 
see e.g. \cite{revuz2013continuous,KS,RW2}, Bessel processes play a prominent role as a fundamental 
example on which the extraordinary power of the theory can be tested. 

Stochastic partial differential equations (SPDEs) were invented around fifty years ago as a natural function-valued analog of SDEs, and are by now a well-established field which is increasingly active
and lively. SPDEs driven by a space-time white noise have recently received much attention, because
they are naturally associated with {\it ultraviolet divergences} and {\it renormalisation}, phenomena
which are now mathematically well-understood in many circumstances using the recent theories of regularity structures
\cite{Hairer2014d,BHZ} and of paracontrolled distributions \cite{GIP}. 

In particular, the classical stochastic calculus for semimartingales and SDEs has no analog for
space-time-white-noise driven SPDEs, despite some early and more recent attempts \cite{Zambotti2006,Bellingeri}, because of the divergences created by the white noise. A partial
substitute is given by the Fukushima stochastic calculus associated with Dirichlet forms \cite{fukushima2010dirichlet,ma2012introduction}, but the formulae that one obtains are often less explicit than one would hope. The marvellous power of the It\^o calculus for the study of fine properties of semimartingales remains without proper analog in genuinely infinite-dimensional processes.

In this paper we discuss a particular class of equations which seems a natural analog of Bessel
processes in the context of SPDEs driven by a space-time white noise. As we explain below, the 
standard approach to Bessel processes does not work at all for these Bessel SPDEs, and we have to
apply a different method, with necessarily weaker results, at least in comparison with the finite-dimensional situation. 
We will rely on Dirichlet forms methods and on integration by parts formulae on path spaces. These will include distributional terms - rather than $\sigma$-finite measures - as in the theory of white noise calculus \cite{hida}.

The processes that we consider have interesting path properties, as it is the case for Bessel processes, but with the enhanced richness of 
infinite-dimensional objects, see e.g. \cite{zambotti2017random} for a recent account. We hope that this work will further motivate the study of
infinite-dimensional stochastic calculus, which is still in its infancy.

\subsection{From Bessel SDEs to Bessel SPDEs}

A squared Bessel process of dimension $\delta\geq 0$ is defined as the unique continuous non-negative process
$(Y_t)_{t\geq 0}$ solving the SDE
\begin{equation}\label{sqB1}
Y_t=Y_0+\int_0^t2\sqrt{Y_s} \d B_s +\delta \,t, \quad t\geq 0, \qquad (\delta\geq 0)
\end{equation}
for $Y_0 \geq 0$, where $(B_t)_{t\geq 0}$ is a standard Brownian motion. Squared Bessel processes enjoy a remarkable additivity property (see \cite{shiga1973bessel} and \eqref{additivity_sqred_bes} below), and play a prominent in several areas of probability theory. For instance, in population dynamics, they arise as the scaling limit of Galton-Watson processes with immigration. On the other hand, they play an important role in the
study of the fine properties of Brownian motion, see e.g. the sections VI.3 and XI.2 in \cite{revuz2013continuous}. Moreover their fascinating behavior at the boundary point 0 can be studied in great detail, see e.g. \cite{zambotti2017random} for a recent account.

Well-posedness of the SDE \eqref{sqB1} satisfied by $(Y_t)_{t\geq 0}$ follows from the classical Yamada-Watanabe theorem \cite[Theorem IX.3.5]{revuz2013continuous}. If we consider the {\it Bessel process} $X_t:=\sqrt{Y_t}$, $t\geq 0$, the situation is more involved. For $\delta>1$, by the It\^o formula, $X$ is solution to
\begin{equation}\label{sde1}
X_t=X_0+\frac{\delta-1}2\int_0^t \frac1{X_s}\d s+ B_t, \quad t\geq 0, \qquad (\delta>1)
\end{equation}
and this equation satisfies pathwise uniqueness and existence of strong solutions since the drift is monotone decreasing, see Prop 3.1 in \cite{zambotti2017random} or section V.48 in \cite{RW2}. On the other hand, for $\delta=1$, $X=\sqrt{Y}$ satisfies
\[
X_t=X_0+L_t+ B_t, \qquad \quad t\geq 0, \qquad (\delta=1)
\]
where $(L_t)_{t\geq 0}$ is continuous and monotone non-decreasing, with $L_0=0$ and
\begin{equation}\label{sde2}
X\geq 0, \qquad \int_0^\infty X_s \d L_s=0.
\end{equation}
In other words $X$ is a reflecting Brownian motion, and the above equation has a unique solution by the Skorokhod Lemma \cite[Lemma VI.2.1]{revuz2013continuous}.

For $\delta\in (0,1)$, the situation is substantially more difficult and it turns out that the relation \eqref{sde1} is not valid anymore in this regime. 
One can show, see e.g. \cite[Proposition 3.12]{zambotti2017random}, that 
$X$ admits {\it diffusion local times}, namely continuous processes $(\ell^a_t)_{t\geq 0,a\geq 0}$ such that
\begin{equation}\label{otfo}
\int_0^t \varphi(X_s)\d s =\int_0^\infty \varphi(a) \, \ell^a_t \, a^{\delta-1} \d a,
\end{equation}
for all Borel $\varphi:\mathbb{R}_+\to\mathbb{R}_+$, and that $X$ satisfies
\begin{equation}\label{sde3}
X_t=X_0+\frac{\delta-1}2\int_0^\infty\frac{\ell^a_t-\ell^0_t}a \, a^{\delta-1}\d a+B_t, \quad t\geq 0, \qquad (0<\delta<1).
\end{equation}
Note that by the occupation time formula \eqref{otfo} we have
\[
\begin{split}
&\int_0^\infty\frac{\ell^a_t-\ell^0_t}a \, a^{\delta-1}\d a = \lim_{\varepsilon\downarrow 0}
\int_\varepsilon^\infty\frac{\ell^a_t-\ell^0_t}a \, a^{\delta-1}\d a =
\\ & = \lim_{\varepsilon\downarrow 0} \left(\int_0^t \un{(X_s\geq\varepsilon)} \frac1{X_s}\d s 
- \ell^0_t \int_\varepsilon^\infty a^{\delta-2}\d a\right)
\end{split}
\]
and in the latter expression both terms diverge as $\varepsilon\downarrow 0$, while the difference
converges since $|\ell^a_t-\ell^0_t|\lesssim a^{1-\frac\delta2-\kappa}$ for any $\kappa>0$: this is why we speak of {\it renormalised local times}. 

The formula \eqref{sde3} is not really an SDE, and to our knowledge one cannot (so far) characterize $X$ 
as the unique process satisfying this property, unless one manages to prove that $X^2$ is a solution
to \eqref{sqB1}. 
We stress again that the relation between \eqref{sqB1} and \eqref{sde1}-\eqref{sde2}-\eqref{sde3} is based on It\^o's stochastic
calculus. 

In a series of papers \cite{Z01,zambotti2002integration,zambotti2003integration,zambotti2004occupation} the second author of this article studied a class of stochastic partial 
differential equations (SPDEs) with analogous properties. For a parameter $\delta>3$ the equation, that we call {\it Bessel SPDE}, is
\begin{equation}\label{spde>3}
\left\{ \begin{array}{ll}
{\displaystyle
\frac{\partial u}{\partial t}=\frac 12
\frac{\partial^2 u}{\partial x^2}
 + \frac {\kappa(\delta)}{2 \, u^3} 
 + \xi 
}
\\ \\
u(0,\cdot)=u_0, \ u(t,0)=u(t,1)=0
\end{array} \right.
\qquad \qquad (\delta>3)
\end{equation}
where $u\geq 0$ is continuous and $\xi$ is a space-time white noise on $\mathbb{R}_+\times[0,1]$, and
\begin{equation}\label{kappadelta}
\kappa(\delta) := \frac{(\delta-3)(\delta-1)}{4}.
\end{equation}
As $\delta\downarrow 3$, the solution to \eqref{spde>3} converges to the solution of the Nualart-Pardoux equation \cite{nualart1992white}, namely the random obstacle problem
\begin{equation}\label{spde=3}
\left\{ \begin{array}{ll}
{\displaystyle
\frac{\partial u}{\partial t}=
\frac 12\frac{\partial^2 u}{\partial x^2}
 + \eta+ \xi }
\\ \\
u(0,\cdot)=u_0, \ u(t,0)=u(t,1)=0
\\ \\
u\geq 0, \ d\eta\geq 0, \
\int_{\mathbb{R}_+\times[0,1]} u\, \d\eta=0,
\end{array} \right.
\qquad \qquad (\delta=3)
\end{equation}
where $\eta$ is a Radon measure on $]0,\infty[\,\times\,]0,1[$. The unique invariant measure of \eqref{spde>3} for $\delta>3$, respectively \eqref{spde=3}, is the Bessel bridge of dimension $\delta$, resp. 3. In other words, the invariant measure has the law
of $(X_t)_{t\in[0,1]}$ conditioned to return to 0 at time 1, where $X$ solves \eqref{sde1} with $X_0=0$ and $\delta>3$, respectively $\delta=3$.
Equation \eqref{spde=3} also describes the fluctuations of an effective $(1+1)$ interface model near a wall \cite{funakiolla,funakistflour} and also arises as the scaling limit of several weakly asymmetric interface models, see \cite{etheridge2015scaling}. 

While \eqref{spde>3} for $\delta>3$ is the analog of \eqref{sde1} for $\delta>1$, \eqref{spde=3} is the 
analog of \eqref{sde2}. The analogy can be justified in terms of scaling invariance: the equations
\eqref{sde1} and \eqref{sde2} are invariant (in law) under the rescaling $X_t\mapsto \lambda^{-1} X_{\lambda^2t}$ for $\lambda>0$, while \eqref{spde>3} and \eqref{spde=3} are invariant under $u(t,x)\mapsto\lambda^{-1} u(\lambda^4t,\lambda^2x)$ (apart from the fact that the space interval changes
from $[0,1]$ to $[0,\lambda^{-2}]$).

It has been an open problem for over 15 years to complete the above picture. Namely, what is an SPDE whose invariant measure is the Bessel bridge of dimension $\delta\in\,]0,3[$  ? Is it an SPDE analogue of \eqref{sde3} ? 

We stress that equations \eqref{spde>3} and \eqref{spde=3} enjoy nice properties (pathwise uniqueness, continuity with respect to initial data, the Strong Feller property) because
of the {\it dissipative}, namely monotone non-increasing, character of the drift. This is however 
true only as long as the coefficient $\kappa(\delta)$ is positive, and fails 
for $\delta\in\,]1,3[$. In the regime $\delta<3$, we shall see that even the notion of solution
becomes highly non-trivial, as for Bessel processes in the regime $\delta<1$. The nice properties mentioned above may still be true but the known 
techniques become ineffective.

This problem is particularly interesting for $\delta=1$, which corresponds to the reflecting Brownian bridge as an invariant measure. Indeed, the reflecting Brownian bridge arises as the scaling limit of critical pinning models, see \cite{dgz}, \cite[Chapter 15.2]{funakistflour} and \cite{fattler2016construction,grothaus18feller}.
Dynamical pinning models are believed to have a scaling limit, which would be an infinite-dimensional diffusion having the law of a reflecting Brownian motion as reversible measure. What kind of SPDE that limit should satisfy has however remained a very open question so far. Another application of a Bessel SPDE corresponding to $\delta=1$ could be the description of the scaling limits of the spin flip dynamics considered in \cite{caputo2008approach}. 

Note that the one-dimensional trick of considering a power of $u$, in this case for instance $v:=u^4$, in order to find a more tractable SPDE fails
because one obtains rather frightening equations of the form
\[
\frac{\partial v}{\partial t}=\frac 12
\frac{\partial^2 v}{\partial x^2}
 + 2\,\kappa(\delta) - \frac3{8 \, v} \, : \left(\frac{\partial v}{\partial x}\right)^2: 
 + 4v^{\frac34}\, \xi 
\]
where the $: \ :$ notation denotes a KPZ-type renormalisation. Even the theory of regularity structures 
\cite{Hairer2014d} does not cover this kind of equations,
due to the non-Lipschitz character of the coefficients. One could hope that a Yamada-Watanabe result could be proved for this class of equations;
it is an inspiring fact that the exponent $\frac34$ in the noise-term is known to be critical for pathwise uniqueness of parabolic SPDEs (without the KPZ-type term), see \cite{mytnik1,mytnik2}. This approach is, at present, completely out of reach.

Therefore, in this paper, rather than tackling these difficulties, we answer the above questions by exploiting the specific, very nice structure underlying Bessel processes. More precisely, we derive integration by parts formulae for the law of Bessel bridges of dimension $\delta<3$. These formulae turn out to involve the laws of pinned Bessel bridges (or, more precisely, the measures $\Sigma^\delta_r(\,\cdot\,|\,a)$ defined in \eqref{Sigma} below) which should correspond to the local times of the solution $u$ to our would-be SPDEs, that is the process $(\ell^a_{t,x})_{a \geq 0}$ defined, at least formally, by
\begin{equation}
\label{otf}
\int_0^t \varphi(u(s,x))\d s =\int_0^\infty \varphi(a) \, \ell^a_{t,x} \, a^{\delta-1} \d a,
\end{equation}
for all Borel $\varphi:\mathbb{R}_+ \to \mathbb{R}_+$. Some explicit computations on the measures $\Sigma^\delta_r(\,\cdot\,|\, a)$ suggest that this process should moreover have a vanishing first-order derivative at $0$, that is
\begin{equation}
\label{vanishing_derivative}
\frac{\partial}{\partial a} \ell^{a}_{t,x}\, \biggr\rvert_{a=0} = 0, \quad \quad t \geq 0, \quad x \in (0,1). 
\end{equation}
Finally, the integration by parts formulae that we find enable us to identify the structure of the corresponding Bessel SPDEs. Thus,  in view of these formulae, for $1<\delta<3$, the SPDE should have the form
\begin{equation}\label{1<spde<3}
\frac{\partial u}{\partial t}=\frac 12
\frac{\partial^2 u}{\partial x^2}
 + \frac {\kappa(\delta)}{2}\frac\partial{\partial t}\int_0^\infty \frac1{a^3}\left(\ell^a_{t,x}-\ell^0_{t,x}\right)  a^{\delta-1}\d a  + \xi, \qquad (1<\delta<3).
\end{equation}
Then \eqref{1<spde<3} is the SPDE analog of \eqref{sde3}. On the other hand, for $\delta=1$, we find that the SPDE should be of the form
\begin{equation}\label{spde=1}
\frac{\partial u}{\partial t}=\frac 12
\frac{\partial^2 u}{\partial x^2}
 - \frac{1}{8} \frac\partial{\partial t}\frac{\partial^{2}}{\partial a^{2}} \, \ell^{a}_{t,x}\, \biggr\rvert_{a=0} + \xi, \qquad (\delta=1),
\end{equation}
while for $0<\delta<1$
\begin{equation}\label{0<spde<1}
\begin{split}
& \frac{\partial u}{\partial t}=\ \frac 12
\frac{\partial^2 u}{\partial x^2} + \xi
\qquad\qquad\qquad\qquad\qquad\qquad\qquad (0<\delta<1)
\\ & + \frac {\kappa(\delta)}{2}\frac\partial{\partial t}\int_0^\infty \frac1{a^3}\left(\ell^a_{t,x}-\ell^0_{t,x}-\frac{a^2}2\frac{\partial^{2}}{\partial a^{2}} \, \ell^{a}_{t,x}\, \biggr\rvert_{a=0}\right)  a^{\delta-1}\d a.
\end{split}
\end{equation}
In \eqref{1<spde<3}, as in \eqref{sde3}, we have a Taylor expansion at order 0 of the local times functions 
$a\mapsto\ell^a$. By contrast, equations \eqref{spde=1} and \eqref{0<spde<1} have no analog in the context of 
one-dimensional Bessel processes. In \eqref{0<spde<1} the Taylor expansion is at order 2, while \eqref{spde=1} is 
a limit case, like \eqref{spde=3}. 

In all the above, we say "the SPDE should have the form..." since existence and uniqueness of solutions to such equations are still
open problems, as we discuss below. In the case $\delta=1$, we show below that our integration by parts formula and Dirichlet
forms techniques allow to construct a Markov process $(u_t)_{t \geq 0}$ with the reflected Brownian bridge as reversible measure, and satisfying a modified version of equation \eqref{spde=1} above, namely
\begin{equation}\label{formal1}
\frac{\partial u}{\partial t}=\frac 12
\frac{\partial^2 u}{\partial x^2}
 - \frac{1}{4} \, \underset{\epsilon \to 0}{\lim} \, \rho''_{\epsilon}(u) + \xi, 
 \qquad (\delta=1)
\end{equation}
where $\rho_{\epsilon}(x) = \frac{1}{\epsilon} \rho(\frac{x}{\epsilon})$ is a smooth approximation of the Dirac measure at $0$, see Theorem \ref{fukushima_decomposition} for the precise statements. Note that
\eqref{formal1} is a weak form of \eqref{spde=1}, since the former does not require the existence of the local
time process $\ell^a$. Similar arguments allow us to treat the case $\delta=2$: this will be done in a forthcoming article.
In the cases $\delta\in\,]0,3[\,\setminus\{1,2\}$, we do not know how to prove that the associated Dirichlet form is well-defined and
associated with a Markov process (namely that it is closable and quasi-regular); once this is done, our integration by parts formulae
allow to show that the associated Markov process satisfies \eqref{1<spde<3} or \eqref{0<spde<1} according to the value of $\delta$.

Although they seem quite different, all the above SPDEs can be written in a unified way as follows. 
We introduce for $\alpha\in\R$ the following distributions on $[0,\infty)$
\begin{itemize}
\item if $ \alpha = -k$ with $k \in \mathbb{N}$, then 
\[
\langle \mu_{\alpha}, \varphi \rangle := (-1)^{k} \varphi^{(k)}(0), \qquad \forall \,
\varphi \in C^\infty_0([0,\infty))
\]
\item else, 
\[
\langle \mu_{\alpha} , \varphi \rangle := \int_{0}^{+ \infty} \left(
\varphi(a) - \sum_{0\leq j\leq 
-\alpha} \frac{a^{j}}{j!} \, \varphi^{(j)}(0) \right) \frac{a^{\alpha -1}}{\Gamma(\alpha)} \d a, \quad \forall \,
\varphi \in C^\infty_0([0,\infty)).
\]
\end{itemize}
Note that, for all $\alpha \in \mathbb{R}$, $\mu_\alpha$ coincides with the distribution $\frac{x_{+}^{\alpha-1}}{\Gamma(\alpha)}$ considered in Section 3.5 of \cite{gelfand1964generalized}. Then for all $\varphi \in C^\infty_0([0,\infty))$, the map $\alpha\mapsto\langle \mu_{\alpha} , \varphi \rangle$ is analytic. Moreover, for $\delta>3$, the non-linearity in \eqref{spde>3} can be expressed by the occupation time formula \eqref{otf} and the definition of $\mu_\alpha$ as
\[\frac{\kappa(\delta)}2\int_0^t\frac1{(u(s,x))^3}\d s = \frac{\kappa(\delta)}2\int_0^\infty \frac1{a^3} \, \ell^a_{t,x}\, a^{\delta-1}\d a = \frac{\kappa(\delta)\,\Gamma(\delta-3)}2\langle \mu_{\delta-3},\ell^{\boldsymbol{\cdot}}_{t,x}\rangle,\]
which, by \eqref{kappadelta}, we can in turn rewrite as
\begin{equation}
\label{unified_expr_nonlinearity}
\frac{\Gamma (\delta)}{8(\delta-2)} \langle \mu_{\delta-3},\ell^{\boldsymbol{\cdot}}_{t,x}\rangle,
\end{equation}
an expression which, at least formally, makes sense for any $\delta \in (0,\infty) \setminus \{2\}$. 
Note moreover that the singularity at $\delta=2$ is compensated by the cancellation of $\langle \mu_{\delta-3},\ell^{\boldsymbol{\cdot}}_{t,x}\rangle$ at $\delta=2$ as a consequence of \eqref{vanishing_derivative}.
Then, the expression \eqref{unified_expr_nonlinearity} encapsulates, in a unified way, the non-linearities of \eqref{spde=3}-\eqref{1<spde<3}-\eqref{spde=1}-\eqref{0<spde<1}. In particular, for $\delta=3$, it equals $\frac{1}{4}  \ell^0_{t,x}$, which is consistent with the results about the structure of the reflection
measure $\eta$ in \eqref{spde=3} proved in \cite{zambotti2004occupation} and showing that a.s. 
\[
\eta([0,t]\times{\rm d}x) = \frac14\, \ell^0_{t,x}  \d x.
\]
At least formally, the $\delta$-Bessel SPDEs for $\delta <3$ correspond to the unique analytic continuation of the $\delta$-Bessel SPDEs for $\delta \geq 3$. This is justified by considering the corresponding integration by parts formulae on a specific set of test functions, where every term depends in an analytic way on $\delta$, see \eqref{exp_fst_part_ibpf0} below.

\subsection{Integration by parts formulae for the laws of Bessel bridges}

Integration by parts plays a fundamental role in analysis, and most notably in stochastic analysis. For instance, it lies at the core of Malliavin Calculus and the theory of Dirichlet forms, see e.g. \cite{nualart,fukushima2010dirichlet,ma2012introduction}. 

While it is relatively easy in finite dimension, where the standard rules of calculus apply, obtaining integration by parts formulae (IbPFs for short) for measures on infinite-dimensional spaces can be a difficult task, one of the main reasons being the absence of Lebesgue measure in that context. The most celebrated example is the IbPF associated with Brownian motion, or its corresponding bridge, on the interval $[0,1]$, which reads
\[ E \left[\partial_{h} \Phi (B) \right] = -  E \left[\langle h'', B \rangle \, \Phi (B) \right], \] 
for all Fr\'{e}chet differentiable $\Phi : L^{2}(0,1) \to \mathbb{R}$ and all $h \in C^{2}_{c}(0,1)$, where $ \langle \cdot, \cdot \rangle$ denotes the canonical scalar product in $L^2(0,1)$. This formula follows for instance from the quasi-invariance property of the Wiener measure on $[0,1]$ along the Cameron-Martin space, 
by differentiating at $\gep=0$ the formula
\[
E[ \Phi(B+\gep h)] = E\left[\Phi(B)\, \exp\left(-\gep\langle h'',B\rangle -\frac{\gep^2}2\|h'\|_{L^2(0,1)}^2\right) \right].
\]

In \cite{zambotti2002integration}, the second author exploited the relation between the law of the 
Brownian bridge and the law $P^3$ of the $3$-dimensional Bessel bridge (also known as the normalised Brownian excursion) on $[0,1]$ to deduce an IbPF for the latter measure; other proofs were given later, see e.g. \cite{FuIs,zambotti2017random}. In \cite{zambotti2003integration}, exploiting an absolute continuity relation with respect to the $3$-dimensional Bessel bridge, the second author obtained IbPFs for the law $P^\delta$ of 
Bessel bridges of dimension $\delta>3$. Put in a nutshell, these formulae read as follows:
\begin{equation}
\label{ibpf_larger_three}
E^{\delta} \left[\partial_{h} \Phi (X) \right] + E^{\delta}  \left[\langle h'', X \rangle \, \Phi (X) \right] = - \kappa(\delta) \, E^{\delta}  \left[\langle h, X^{-3} \rangle \, \Phi (X) \right]
\end{equation}
for all $\delta >3$, and
\begin{equation}\label{ibpf_three}
\begin{split}
& E^{3} \left[\partial_{h} \Phi (X) \right] +  E^{3}  \left[\langle h'', X \rangle \, \Phi (X) \right] = 
\\ & = - \int_{0}^{1} \d r \, \frac{h_r}{\sqrt{2\pi r^3(1-r)^3}} \, E^{3}  \left[\Phi (X) \, | \, X_{r}=0\right], 
\end{split}
\end{equation}
where $\Phi$ and $h$ are as above. Here, for all $\delta >0$, $E^{\delta}$ denotes the expectation with respect to the law $P^{\delta}$, on the space of continuous real-valued functions on $[0,1]$, of the $\delta$-dimensional Bessel bridge from $0$ to $0$ over the interval $[0,1]$, and $\kappa(\delta)$ is defined in \eqref{kappadelta}.
Note that while $\kappa(\delta) > 0$ for $\delta > 3$, $\kappa$ vanishes at $\delta=3$, the dimension corresponding to the Brownian excursion. At the same time, the quantity 
$\langle |h|, X^{-3} \rangle $
is integrable with respect to $P^{\delta}$ for $\delta > 3$, but is non-integrable with respect to $P^{3}$ for $h$ that is not identically $0$. Thus, when $\delta \searrow 3$, the right-hand side of \eqref{ibpf_larger_three} is an indeterminate form which turns out to converge to the non-trivial quantity in the right-hand side of \eqref{ibpf_three}; this can be seen, at least for Fr\'echet differentiable $\Phi$, by comparing the
left-hand sides of the two formulae and by using continuity of the map $\delta\mapsto P^\delta$. Formula \eqref{ibpf_three} also possesses a geometric-measure theory interpretation as a Gauss-Green formula in an infinite-dimensional space, the second term in the right-hand side corresponding to a boundary term (see Chapter 6.1.2 in \cite{zambotti2017random}). 

What can we say for Bessel bridges of dimension $\delta <3$ ? In such a regime, the techniques used in \cite{zambotti2003integration}, based on absolute continuity relations with the Brownian excursion as well as monotonicity arguments, fall apart. Indeed, when $\delta \in (1,3)$, $\kappa(\delta) <0$, so the required monotonicity properties do not hold anymore, while for $\delta <2$ the absolute continuity relations fail to exist. Hence, the problem of finding IbPFs for the measures $P^{\delta}$, when $\delta <3$, has remained open until now, excepted for the value $\delta=1$, corresponding to the reflected Brownian bridge, for which some (strictly weaker) IbPFs have been obtained, see \cite{zambotti2005integration} for the case of the reflected Brownian motion, \cite{grothaus2016integration} for the case of a genuine bridge,
and Remark \ref{weaker} below for a discussion. 

\subsection{Outline of the results}

Here and below, let $C([0,1]) := C([0,1], \mathbb{R})$ be the space of continuous real-valued functions on $[0,1]$.  In this article, we obtain IbPFs for the laws $P^{\delta}$ of Bessel bridges of dimension $\delta \in (0,3)$ from $0$ to $0$ over $[0,1]$. Our formulae hold for a large class of functionals $\Phi: C([0,1]) \to \mathbb{R}$. More precisely, we consider linear combinations of functionals of the form 
\begin{equation}\label{suitable} \Phi(\zeta) =  \exp(- \langle m, \zeta^{2} \rangle), \quad \zeta \in C([0,1]), \end{equation}
with $m$ a finite Borel measure on $[0,1]$, and where $\langle m, \zeta^{2} \rangle := \int_0^1 \zeta_t^{2} \, m({\rm d} t)$. We prove that these functionals satisfy IbPFs for the laws $P^{\delta}$, for all $\delta >0$. Our method is based on deriving semi-explicit expressions for quantities of the form
\[  E^{\delta}  \left[\Phi (X) \right] \qquad
\text{and} \qquad
E^{\delta} \left[\Phi (X) \, | \, X_{r} = a\right], \quad a \geq 0,  \, r \in (0,1), \]
using solutions to some second-order differential equations, and exploiting the nice computations done in Chapter XI of  \cite{revuz2013continuous}. The fundamental property enabling these computations is the additivity property of the squared Bessel processes, which in particular implies that both of the quantities above factorize in a very specific way, see the expression \eqref{bridge2} below. As a consequence, for functionals as above, all the IbPFs for $P^{\delta}$, $\delta \geq 3$ are just multiples of a single differential relation which does not depend on $\delta$ (see Lemma \ref{thm} below), the dependence in $\delta$ entering only through the multiplying constant which involves some $\Gamma$ values. When $\delta \geq 3$, expressing these $\Gamma$ values as integrals, and performing a change of variable, we retrieve the formulae already obtained in \cite{zambotti2002integration} and \cite{zambotti2003integration}. On the other hand, when $\delta <3$, one of the $\Gamma$ values appearing is negative, so we cannot express it using the usual integral formula, but must rather use \textit{renormalised} integrals. 

As a result, when $\delta \in (1,3)$, the IbPFs can be written
\begin{equation}
\label{new_ibpf_13}
\begin{split}
& E^{\delta} (\partial_{h} \Phi (X) ) + E^{\delta} (\langle h '' , X \rangle \, \Phi(X) ) = \\
&=-\kappa(\delta)\int_{0}^{1}  
 h_{r}  \int_0^\infty a^{\delta-4} \Big[ \Sigma^\delta_r \left(\Phi (X) \, | \, a\right) - \Sigma^\delta_r  \left(\Phi (X) \, | \, 0\right) \Big]
  \d a \d r,
\end{split}
\end{equation}  
where, for all $a \geq 0$, $\Sigma^\delta_r \left({\rm d}X \, | \, a\right)$ is a measure on $C([0,1])$ proportional to the law of the Bessel bridge conditioned to hit $a$ at $r$, see \eqref{Sigma}. Thus, the left-hand side is the same as for \eqref{ibpf_larger_three} and \eqref{ibpf_three}, but the right-hand side now  contains Taylor remainders at order $0$ of the functions $a \mapsto \Sigma^\delta_r \left(\Phi (X) \, | \, a\right)$. When $\delta \in (0,1)$, this renormalisation phenomenon becomes even more acute. Indeed, in that case, the IbPFs are similar to \eqref{new_ibpf_13}, but the right-hand side is replaced by
\begin{equation}
\label{new_ibpf_01} 
-\kappa(\delta)\int_{0}^{1}  
 h_{r}  \int_0^\infty a^{\delta-4} \left[\varphi(a)-\varphi(0)-\frac{a^2}2\varphi''(0)
 \right]   \d a \d r, 
\end{equation}
where $\varphi(a):=\Sigma^\delta_r  \left(\Phi (X) \, | \, a\right)$, and where we see Taylor remainders at order 2 appearing. An important remark is that the terms of order 1 vanish
\[
\varphi'(0)=\left. \frac{\rm d}{{\rm d} a} \Sigma^\delta_r \left(\Phi (X) \, | \, a \right) \, \right|_{a=0} = 0 , \quad r \in (0,1),
\]
so we do not see them in the above Taylor remainders.  Finally, in the critical case $\delta=1$, we obtain the fomula
\begin{equation}
\label{new_ibpf_1}
\begin{split}
E^{1} (\partial_{h} \Phi (X) ) + E^{1} (\langle h '' , X \rangle \, \Phi(X) ) = 
\frac{1}{4} \int_{0}^{1} h_{r} \, \frac{{\rm d}^{2}}{{\rm d} a^{2}} \, \Sigma^1_r (\Phi(X) \, | \, a) \, \biggr\rvert_{a=0}  \d r. 
\end{split}
\end{equation}

The IbPFs are stated in Theorem \ref{statement_ibpf} below. One important, expected feature is the transition that occurs at the critical values $\delta=3$ and $\delta=1$. Another important but less expected feature is the absence of transition at $\delta =2$, as well as the related remarkable fact that the functions $a \mapsto  \Sigma^\delta_r \left(\Phi (X) \, | \, a\right)$ are, for all $r \in (0,1)$, smooth functions in $a^{2}$, so that all their odd-order derivatives vanish at $0$. This is the reason why there only ever appear derivatives of even order in our formulae. An objection to this observation might be that the class of functionals \eqref{suitable} is too restrictive. However, in a forthcoming article, we will show that the IbPFs obtained in the present article still hold for a class of very general functionals. In particular, vanishing of first-order derivatives at $a=0$ can be established for $ a \mapsto \Sigma^\delta_r \left(\Phi (X) \, | \, a\right)$ 
for any $\Phi \in C^{1}_{b}(L^{2}(0,1))$, which confirms the absence of transition at $\delta=2$ observed in this article. Finally, note that all the IbPFs above can be written in a unified way, by re-expressing the last term as 
\[-\frac{\Gamma (\delta)}{4(\delta-2)} \int_0^1 \langle \mu_{\delta-3} ,\Sigma^\delta_r (\Phi(X) \, | \, \cdot\,) \rangle \]
in analogy with \eqref{unified_expr_nonlinearity}.
The latter formula bears out the idea that the new IbPFs for Bessel bridges of dimension $\delta<3$ are given by the unique analytic continuation of those for $\delta \geq 3$, at least for suitable test functionals $\Phi$ as in \eqref{suitable}.

The IbPFs \eqref{new_ibpf_13}, \eqref{new_ibpf_01} and \eqref{new_ibpf_1} above suggest that the gradient dynamics associated with the laws of Bessel bridges of dimension $\delta <3$ should be given by the SPDEs \eqref{1<spde<3}, \eqref{0<spde<1} and \eqref{spde=1} respectively. Note that, in the case $\delta \geq 3$, the SPDEs had been solved in \cite{nualart1992white,zambotti2003integration} using pathwise techniques, and many fine properties of the solution had been studied, such as their hitting properties (see \cite{dalang2006hitting}), or the existence of occupation densities (see \cite{zambotti2004occupation}). By contrast, in the case $\delta < 3$, the SPDEs \eqref{1<spde<3}, \eqref{spde=1} and \eqref{0<spde<1} do not yet seem to possess any strong notion of solution, and essentially lie outside the scope of any existing theory of SPDEs. However, in this article, for $\delta=1$, using Dirichlet form techniques, and thanks to the IbPF \eqref{new_ibpf_1} for the reflecting Brownian bridge, we are able to construct a weak version of the associated SPDE in the stationary regime. Thus, the dynamics for $\delta=1$ can be described by \eqref{formal1}, which is a weaker version of \eqref{spde=1}. We also prove (see Theorem \ref{dist_res} below) that the corresponding Markov process does not coincide with the process associated with the absolute value of the solution to the stochastic heat equation. A similar construction can be implemented in the case $\delta=2$: this will be done in the forthcoming article \cite{henri2018bessel}. The approach using Dirichlet forms was already used in Robert Vo{\ss}hall's thesis \cite{vosshallthesis}, which provided a construction of the Markov process for $\delta=1$, but not the SPDE.   

\medskip
The article is organized as follows: in Section \ref{sect_prelude} we address a toy-model consisting of a family of measures on $\mathbb{R}_{+}$, hence much simpler than the laws of Bessel bridges, but displaying a similar renormalisation phenomenon at the level of the IbPFs. In Section \ref{sect_sqred_bessel} we recall and prove some useful facts on the laws of squared Bessel processes, Bessel processes, and their bridges. In Section \ref{sect_ibpf_exp_func}, we state and prove the IbPFs for the laws of Bessel bridges. 
The dynamics associated with the law of a reflected Brownian bridge is constructed and studied in Section \ref{sect_Dirichlet}. Finally, in Section \ref{sect_conj_dynamics}, we justify our conjectures \eqref{1<spde<3} \eqref{spde=1} and \eqref{0<spde<1} for the $\delta$-Bessel SPDEs for $\delta<3$, and we formulate some additional related conjectures.

\medskip
{\bf Acknowledgements.}
The arguments used in Prop \ref{closability} below to show quasi-regularity of the form associated with the law of a reflected Brownian bridge were communicated to us by Rongchan Zhu and Xiangchan Zhu, whom we warmly thank. The first author is very grateful to Jean-Dominique Deuschel, Tal Orenshtein and Nicolas Perkowski for their kind invitation to TU Berlin, and for very interesting discussions. We also thank Giuseppe Da Prato for very useful discussion and for his kindness and patience in answering our questions. The authors would finally like to thank the Isaac Newton Institute for Mathematical Sciences for hospitality and support during the programme "Scaling limits, rough paths, quantum field theory" when work on this paper was undertaken: this work was supported by EPSRC grant numbers EP/K032208/1 and EP/R014604/1. The second author gratefully acknowledges support by the Institut Universitaire de France and the project of the Agence Nationale de la Recherche ANR-15-CE40-0020-01 grant LSD.

\section{A prelude}
\label{sect_prelude}

In this section we consider a toy model consisting of a family of Schwartz distributions on $\mathbb{R}_{+}$ satisfying nice integration by parts formulae. The content of this section is classical (see e.g. Section 3.5 of \cite{gelfand1964generalized}), but it will serve as a useful finite-dimensional example for the theory to come.
For $\alpha \geq 0$, we set
\[ \mu_{\alpha}({\rm d}x) = \frac{x^{\alpha - 1}}{\Gamma(\alpha)} \d x, \quad \alpha>0, \qquad
\mu_{0} = \delta_{0}, \] 
where $\delta_{0}$ denotes the Dirac measure at $0$. 
A simple change of variable yields the Laplace transform of the measures $\mu_{\alpha}, \, \alpha \geq 0$
 \begin{equation}
\label{laplace0}
 \int_0^{+\infty} \exp ( - \lambda x) \, \mu_{\alpha}({\rm d}x) = \lambda^{-\alpha}, \qquad \lambda>0, \ \alpha \geq  0. 
\end{equation}

It turns out that the family of measures $(\mu_{\alpha})_{\alpha \geq 0}$ can be extended in a natural way to a family of \textit{distributions} $(\mu_{\alpha})_{\alpha \in \mathbb{R}}$ .
We first define the appropriate space of test functions on $[0,\infty)$.
\begin{df}
Let $S([0,\infty))$ be the space of $C^\infty$ functions $\varphi: [0,\infty) \to \mathbb{R}$ such that, for all $k, l \geq 0$, there exists $C_{k,\ell} \geq0$ such that
\[
 | \varphi^{(k)} (x)  | \, x^{\ell} \leq C_{k,\ell}, \qquad \forall x \geq 0.
\]
\end{df}

For $\alpha < 0$, we will define $\mu_{\alpha}$ as a distribution, using a \textit{renormalisation} procedure based on Taylor polynomials. To do so, for any smooth function $\varphi: \mathbb{R}_{+} \to \mathbb{R}$, for all $n \in \mathbb{Z}$, and all $x \geq 0$, we set
\begin{equation}\label{eq:taylor} \mathcal{T}^{\,n}_{x} \varphi := \varphi(x) - \sum_{0\leq j\leq n} \frac{x^{j}}{j!} \, \varphi^{(j)}(0). \end{equation}
In words, if $n \geq 0$ then $\mathcal{T}^{\,n}_{x} \varphi$ is the Taylor remainder based at $0$, of order $n+1$, of the function $\varphi$, evaluated at $x$; if $n<0$ then $\mathcal{T}^{\,n}_{x} \varphi$ is simply the value of $\varphi$ at $x$.

\begin{df}
\label{def_mu_alpha}
For $\alpha < 0$, we define the distribution  $\mu_{\alpha}$ as follows
\begin{itemize}
\item if $ \alpha = -k$ with $k \in \mathbb{N}$, then 
\begin{equation}
\label{mu_neg_int} 
\langle \mu_{\alpha}, \varphi \rangle := (-1)^{k} \varphi^{(k)}(0), \qquad \forall \,
\varphi \in \mathcal{S}([0,\infty))
\end{equation}
\item if $ - k - 1 < \alpha < -k$ with $k \in \mathbb{N}$, then 
\begin{equation} 
\label{mu_neg_delta}
\langle \mu_{\alpha} , \varphi \rangle := \int_{0}^{+ \infty} \mathcal{T}^{\,k}_{x} \varphi \, \frac{x^{\alpha -1}}{\Gamma(\alpha)} \d x, \qquad \forall \,
\varphi \in \mathcal{S}([0,\infty)).
\end{equation}
\end{itemize}
\end{df}
Note that formula \eqref{mu_neg_delta} defines a bona fide distribution on $\mathcal{S}([0,\infty))$. Indeed, by Taylor's theorem, the integrand is of order $x^{k+\alpha}$ near $0$, therefore integrable there, while it is dominated by $x^{k+\alpha-1}$ near $+\infty$, so is integrable at infinity as well. We note that $\mu_\alpha$ is equal to the generalized function $\frac{x_+^{\alpha-1}}{\Gamma(\alpha)}$ of Section 3.5 of \cite{gelfand1964generalized}.

\begin{rk}
Note that for all $\alpha >0$ and all Borel function $\varphi: \mathbb{R}_+ \to \mathbb{R}_+$, the integral $\int_0^\infty \varphi(x) \, \mu_{\alpha}(\d x) $ coincides with $\Gamma(\alpha)^{-1} {\mathcal M}\varphi(\alpha)$, where ${\mathcal M}\varphi(\alpha)$ is the value of the Mellin transform of the function $\varphi$ computed at $\alpha$.
Definition \ref{def_mu_alpha} thus provides an extension of the Mellin transform of a function $\varphi \in \mathcal{S}([0,\infty))$ to the whole real line. In particular, equality \eqref{mu_neg_int} is natural in view of Ramanujan's Master Theorem, which allows to see the successive derivatives at $0$ of an analytic function as the values, for non-positive integers, of the analytic extension of its Mellin transform. We refer to \cite{amdeberhan2012ramanujan} for more details on this theorem. We also stress that the renormalisation procedure used in equation \eqref{mu_neg_delta} to define $\mu_{\alpha}$ for $\alpha <0$ is very natural, and can also be used to extend the domain of validity of Ramanujan's Master Theorem, see Theorem 8.1 in \cite{amdeberhan2012ramanujan}. 
\end{rk}

\begin{rk}
For $k \in \mathbb{N}$ and $\alpha$ such that $-k-1 < \alpha < -k$, and for all $\varphi \in \mathcal{S}([0,\infty))$, we obtain after $k+1$ successive integration by parts the equality:
\[
\langle \mu_{\alpha} , \varphi \rangle := (-1)^{k+1} \int_{0}^{+ \infty} \varphi^{(k+1)}(x) \, \mu_{\alpha+k+1} ({\rm d} x),
\]
which can be interpreted as a variant of the Caputo differential, at order $-\alpha$, of $\varphi$,
see e.g. (1.17) in \cite{gorenflo2008fractional}. 
\end{rk}

We recall the following basic fact, which is easily proven (see e.g. (5) in Section 3.5 of \cite{gelfand1964generalized}). It can be seen as a toy-version of the integration by parts formulae of Theorem \ref{statement_ibpf} below. 
\begin{prop}
\label{thm_ibpf_mu}
For all $\alpha \in \mathbb{R}$ and $\varphi \in \mathcal{S}([0,\infty))$
\[ \langle \mu_{\alpha}, \varphi' \rangle = - \langle \mu_{\alpha-1}, \varphi \rangle. \]
\end{prop}
In particular, for $\alpha\in(0,1)$ we have the measure $\mu_\alpha$ in the left-hand side of the IbPF and the distribution $\mu_{\alpha-1}$ in the right-hand side.

\begin{rk}
\label{laplace_mu_neg}
As a consequence of Proposition \ref{thm_ibpf_mu}, we deduce that the expression \eqref{laplace0} for the Laplace tranform of $\mu_{\alpha}$ remains true also for negative $\alpha$. Indeed, for such $\alpha$, picking $k \in \mathbb{N}$ such that $\alpha + k > 0$, we have, for all $\lambda >0$
\[ \begin{split}
\langle \mu_{\alpha} , e^{-\lambda \cdot} \rangle &= (-1)^{k} \, \langle \mu_{\alpha+k} , \frac{\d^{k}}{\d x^{k}} e^{-\lambda \cdot} \rangle = \lambda^{k} \, \langle \mu_{\alpha+k} , e^{-\lambda \cdot} \rangle = \lambda^{k} \, \lambda^{-\alpha-k} = \lambda^{-\alpha}.
\end{split} \]
\end{rk}

\section{Bessel processes and associated bridges}
\label{sect_sqred_bessel}

In this section we recall and prove some useful facts about squared Bessel processes, Bessel processes, and 
their corresponding bridges. 
We recall that, for all $\alpha\geq 0$, $\theta>0$, $\Gamma(\alpha,\theta)$ denotes 
the Gamma probability law on $\mathbb{R}_{+}$

\[
\Gamma(\alpha,\theta) ({\rm d}x) = \frac{\theta^{\alpha}}{\Gamma(\alpha)} \,x^{\alpha-1} \,e^{-\theta x} \, \mathbf{1}_{x > 0}\d x,
\qquad \Gamma(0,\theta):= \delta_{0}.
\]

\subsection{Squared Bessel processes and Bessel processes} 

For all $x, \delta \geq 0$, denote by $Q^{\delta}_{x}$ the law, on $C(\mathbb{R}_{+}, \mathbb{R}_{+})$, of the $\delta$-dimensional squared Bessel process started at $x$, namely the unique solution to the SDE \eqref{sqB1} with $Y_0=x$, see Chapter XI of \cite{revuz2013continuous}. We denote by $(X_t)_{t\geq 0}$ the canonical process 
\[
X_t:C([0,1])\to\mathbb{R}, \qquad X_t(\omega):=\omega_t, \quad \omega\in C([0,1]).
\]

\begin{df}
For any interval $I \subset \mathbb{R}_{+}$, and any two probability laws $\mu, \nu$ on $C(I, \mathbb{R}_{+})$, let $\mu \ast \nu$ denote the convolution of $\mu$ and $\nu$, i.e. the image of $\mu \otimes \nu$ under the addition map:
\[ C(I,\mathbb{R}_{+}) \times C(I, \mathbb{R}_{+}) \to C(I, \mathbb{R}_{+}), \quad (x,y) \mapsto x+y. \]
\end{df}

The family of probability measures $\left(Q^{\delta}_{x}\right)_{\delta, x \geq 0}$ satisfies the following well-known additivity property, first observed by Shiga and Watanabe in \cite{shiga1973bessel}.

\begin{prop}
\label{levy}
For all $x,x', \delta, \delta'\geq 0$, we have the following equality of laws on $C(\mathbb{R}_{+}, \mathbb{R}_{+}) $
\begin{equation}
\label{additivity_sqred_bes}
 Q^{\delta}_{x} \ast Q^{\delta'}_{x'} = Q^{\delta + \delta'}_{x + x'} \end{equation}
\end{prop}

We recall that squared Bessel processes are homogeneous Markov processes on $\mathbb{R}_{+}$. Exploiting the additivity property \eqref{additivity_sqred_bes}, Revuz and Yor provided, in section XI of \cite{revuz2013continuous}, explicit expressions for their transition densities $\left( q^{\delta}_{t}(x,y) \right)_{t > 0, x,y \geq 0}$. When $\delta >0$, these are given by
\begin{equation}
\label{density_besq_x_pos}
q^{\delta}_{t}(x,y) = \frac{1}{2t} \left( \frac{y}{x} \right)^{\nu/2} \exp\left( - \frac{x+y}{2t} \right) I_{\nu} \left(\frac{\sqrt{xy}}{t} \right),\quad t >0, \ x>0.
\end{equation}
Here, $\nu := \delta/2 -1>-1$ and $I_{\nu}$ is the modified Bessel function of index $\nu$
\[ I_\nu(z) := \sum_{k=0}^\infty \frac{\left(z/2\right)^{2k + \nu}}{k! \, \Gamma(k + \nu +1)}, \qquad z > 0. \]
For $x=0$, we have
\begin{equation}
\label{density_besq_x_zero}
q^{\delta}_{t}(0,y) = (2t)^{-\frac\delta2} \, \Gamma \left( \delta/2 \right)^{-1} y^{\delta/2-1} \exp\left( - \frac{y}{2t} \right),\quad t >0,
\end{equation}
that is
\[q^{\delta}_{t}(0,y) \d y = \Gamma \left(\frac{\delta}{2}, \frac{1}{2t} \right) ({\rm d} y).
\]
We also denote by $P^\delta_x$ the law of the $\delta$-Bessel process, image of $Q^{\delta}_{x^2}$
under the map
\begin{equation} 
\label{sqrt_map}
C(\mathbb{R}_{+}, \mathbb{R}_{+})\ni \omega \mapsto \sqrt{\omega} \in C(\mathbb{R}_{+}, \mathbb{R}_{+})  .
\end{equation}
We shall denote by $\left( p^{\delta}_{t}(a,b) \right)_{t >0, \, a,b \geq 0}$ the transition densities of a $\delta$-Bessel process. They are given in terms of the densities of the squared Bessel process by the relation
\begin{equation}
\label{relation_denisties_bes_besq}
 \forall t > 0, \quad \forall a, b \geq 0, \quad p^{\delta}_{t}(a,b) = 2 \, b \, q^{\delta}_{t}(a^{2},b^{2}). 
\end{equation}

In section XI of \cite{revuz2013continuous}, Revuz and Yor provided semi-explicit expressions for the Laplace transforms of squared Bessel processes (and also the corresponding bridges). Their proof is based on the fact that, for all $\delta, x \geq 0$, and all finite Borel measure $m$ on $[0,1]$, the measure $\exp \left( - \langle m , X \rangle \right) Q^{\delta}_{x}$ possesses a nice probabilistic interpretation, where we use the notation
\[ \langle m , f \rangle := \int_{0}^{1} f(r) \,m({\rm d}r) \]
for any Borel function $f : [0,1] \to \mathbb{R}_+$. This remarkable fact is used implicitly in \cite{revuz2013continuous} (see e.g. the proof of Theorem (3.2) of Chap XI.3), where the authors compute the one-dimensional marginal distributions of this measure. By contrast, in the proof of Lemma \ref{lap_cond_bridge} below, we will need to compute higher-dimensional marginals. As a convenient way to perform such a computation, we will show that 
the measure $\exp \left( - \langle m , X \rangle \right) Q^{\delta}_{x}$ corresponds (up to a normalisation constant) to the image of the measure $Q^{\delta}_{x}$ under a deterministic time change. To prove this fact, we first introduce some notations.
 
Let $m$ be a finite, Borel measure on $[0,1]$. As in Chap. XI of \cite{revuz2013continuous}, we consider the unique solution $\phi:\mathbb{R}_{+}\to\mathbb{R}$ of the following problem
\begin{equation}
\label{phi} 
\begin{cases}
\phi''({\rm d} r) = 2 \mathbf{1}_{[0,1]}(r) \, \phi_{r} \, m({\rm d} r)  \\
\phi_0=1, \ \phi > 0, \ \phi ' \leq 0  \ \text{on} \ \mathbb{R}_{+},
\end{cases}
\end{equation} 
where the first is an equality of measures (see Appendix 8 of \cite{revuz2013continuous} for existence and uniqueness of solutions to this problem). Note that the above function $\phi$ coincides with the function $\phi_{\mu}$ of Chap XI.1 of \cite{revuz2013continuous}, with $\mu := 2 \mathbf{1}_{[0,1]} \, m$.

\begin{lm}
\label{measure_change}
Let $m$ be a finite, Borel measure on $[0,1]$, and let $\phi$ be the unique solution of \eqref{phi}. Then, for all $x , \delta \geq 0$, the measure $R^{\delta}_{x}$ on $C([0,1])$ defined by
\begin{equation}
\label{measure_r_delta}
R^{\delta}_{x} := \exp \left(- \frac{x}{2} \phi'_0 \right) \phi_1^{-\frac\delta2} \ e^{-\langle m, X\rangle} \ Q^{\delta}_{x} 
\end{equation}
is a probability measure, equal to the law of the process
\[ \left( \phi_t^{2} \ Y_{\varrho_t} \right)_{t \in [0,1]}, \]
where $Y \overset{(d)}{=} Q^{\delta}_{x}$ and $\varrho$ is the deterministic time change
\begin{equation}
\label{def_var_rho}
\varrho _t = \int_{0}^{t} \phi_u^{-2} \d u, \quad t \geq 0. 
\end{equation}
\end{lm}

\begin{proof}
We proceed as in the proofs of Theorem (1.7) and (3.2) in Chapter XI of \cite{revuz2013continuous}. Let $x, \delta \geq 0$. Under $Q^{\delta}_{x}$, $M_{t} := X_{t} - \delta t$ is a local martingale, so we can define an exponential local martingale by setting
\[ Z_{t} = \mathscr{E} \left( \frac{1}{2} \int_{0}^{\cdot} \frac{\phi'_s}{\phi_s} \d M_{s} \right)_{t}. \]
As established in the proof of Theorem (1.7)  of \cite{revuz2013continuous}, we have
\[
\begin{split}
Z_{t} &= \exp \left( \frac{1}{2} \left( \frac{\phi'_t}{\phi_t} X_{t} -  \phi'_0 x - \delta \ln \phi_t \right) - \int_{0}^{t} X_{s} \,m({\rm d}s) \right) \\
         &= \exp \left(- \frac{x}{2} \phi'_0 \right) \phi_t^{-\frac\delta2}  \exp \left( \frac{1}{2} \frac{\phi'_t}{\phi_t} X_{t} - \int_0^t X_s \,m({\rm d}s) \right), 
\end{split}
\]         
recalling that the measure $\mu$ considered in \cite{revuz2013continuous} is given in our case by $2 \, \mathbf{1}_{[0,1]} \, m$.  In particular, we deduce that the measure $R^{\delta}_{x}$  defined by \eqref{measure_r_delta} coincides with $Z_{1} Q^{\delta}_{x}$ (note that $\phi'_1=0$ as a consequence of \eqref{phi}).
Moreover, by the above expression, $(Z_{t})_{t\in[0,1]}$ is uniformly bounded by $\exp \left(- \frac{1}{2} \phi'_0 \right) \phi_1^{-\frac\delta2}$, so it is a martingale on $[0,1]$. Hence, $R^{\delta}_{x}$ defines a probability measure. 

There remains to give a description of $R^{\delta}_{x}$. By Girsanov's theorem, under $R^{1}_{x}$, $\left( X_{t} \right)_{t \in [0,1]}$ solves the following SDE on $[0,1]$
\begin{equation}
\label{sde_h_square}
X_{t} = x + 2 \int_{0}^{t} \sqrt{X_{s}} \d B_{s} + 2\int_{0}^{t} \frac{\phi'_s}{\phi_s} \,X_{s} \d s + t. 
\end{equation}
But a weak solution to this SDE is provided by $(H_{t}^{2})_{t \in [0,1]}$, where 
\[ H_{t} := \left(\sqrt{x} + \int_{0}^{t} \phi_s^{-1} \d W_{s} \right) \phi_t , \]
where $W$ is a standard Brownian motion. By strong and therefore weak uniqueness of solutions to equation \eqref{sde_h_square}, see \cite[Theorem IX.3.5]{revuz2013continuous}, we deduce that $X$ is equal in law to the process $(H_{t}^{2})_{t \in [0,1]}$. On the other hand, by L\'{e}vy's characterization theorem \cite[IV.3.6]{revuz2013continuous}, we have 
\[ (H_{t})_{t \in [0,1]} \overset{(d)}{=} \left( \phi_t \,\gamma_{\varrho_t} \right)_{t \in [0,1]},\]
where $\gamma$ is a standard Brownian motion started at $x$. Hence we deduce that
\[ (H_{t}^{2})_{t \in [0,1]} \overset{(d)}{=} \left( \phi_t^{2} \,Y_{\varrho_t} \right)_{t \in [0,1]}, \]
where $Y \overset{(d)}{=} Q^{1}_{x}$. Therefore, under $R^{1}_{x}$, we have
\[ X  \overset{(d)}{=} \left( \phi_t^{2} \,Y_{\varrho_t} \right)_{t \in [0,1]}.\] 
The claim is thus proven for $\delta=1$ and for any $x \geq 0$. Now, by the additivity property \eqref{additivity_sqred_bes} satisfied by $\left(Q^{\delta}_{x} \right)_{\delta,x \geq 0}$, there exist $A, B>0$ such that, for all $x, \delta \geq 0$, and all finite Borel measure $\nu$ on $[0,1]$, we have 
\[ Q^{\delta}_{x} \left[ \exp \left(- \int_{0}^{1}  \phi_t^{2} \,X_{\varrho_t} \,\nu ({\rm d}t) \right) \right] = A^{x} B^{\delta}, \]
which can be proved exactly as Corollary 1.3 in Chapter XI of \cite{revuz2013continuous}. Note now that the family of probability laws $\left(R^{\delta}_{x} \right)_{\delta,x \geq 0}$ satisfies the same additivity property 
\[ \forall \ \delta, \delta ', x , x'  \geq 0, \quad R^{\delta}_{x} \ast R^{\delta'}_{x'} = R^{\delta + \delta'}_{x+x'}. \] 
Hence, there also exist  $\tilde{A}, \tilde{B} >0$ such that, for all $x, \delta \geq 0$, and $\mu$ as above:
\[ R^{\delta}_{x} \left[ \exp \left(-\int_{0}^{1} X_{t} \,\nu ({\rm d}t) \right) \right] = {\tilde{A}}^{x} {\tilde{B}}^{\delta}. \]
By the previous point, evaluating at $\delta=1$, we obtain
\[ \forall x \geq 0, \quad A^{x} B = {\tilde{A}}^{x} \tilde{B}.\]
Hence $A=\tilde{A}$ and $B = \tilde{B}$, whence we deduce that, for all $\delta, x \geq 0$
\[Q^{\delta}_{x} \left[ \exp \left(- \int_{0}^{1}  \phi_t^{2} \,X_{\varrho_t} \,\nu ({\rm d}t) \right) \right] = R^{\delta}_{x} \left[ \exp \left(-\int_{0}^{1} X_{t} \,\nu ({\rm d}t) \right) \right]. \]
Since this holds for any finite measure $\nu$ on $[0,1]$, by injectivity of the Laplace transform, the claimed equality in law holds for all $\delta, x \geq 0$.
\end{proof}

\subsection{Squared Bessel bridges and Bessel bridges}

For all $\delta>0$ and $x, y \geq 0$, we denote by $Q^{\delta}_{x,y}$ the law, on $C([0,1])$, of the $\delta$-dimensional squared Bessel bridge from $x$ to $y$ over the interval $[0,1]$. In other words, $Q^{\delta}_{x,y}$ is the law of of a $\delta$-dimensional squared Bessel bridge started at $x$, and conditioned to hitting $y$ at time $1$. A rigourous construction of these probability laws is provided in Chap. XI.3 of \cite{revuz2013continuous} (see also \cite{pitman1982decomposition} for a discussion on the particular case $\delta=y=0$). 

In the sequel we shall chiefly consider the case $x=y=0$. We recall that if $X \overset{(d)}{=} Q^{\delta}_{0,0}$, then, for all $r \in (0,1)$, the distribution of the random variable $X_{r}$ is given by $\Gamma(\frac{\delta}{2}, \frac{1}{2r(1-r)})$, so it admits the density $q^{\delta}_{r}$ given by:
\begin{equation}
\label{one_pt_density_sqred_bridge_00} 
q^{\delta}_{r}(z) := \frac{z^{\delta/2-1}}{(2r(1-r))^{\frac\delta2} \Gamma(\delta/2)} \exp \left(- \frac{z}{2r(1-r)} \right),  \quad z \geq 0, 
\end{equation} 
see Chap. XI.3 of \cite{revuz2013continuous}. 

In the same way as one constructs the laws of squared Bessel bridges $Q^{\delta}_{x,y}$ for $\delta>0$ and $x , y \geq 0$, one can also construct the laws of Bessel bridges. In the following, for any $\delta>0$ and $a, b \geq 0$, we shall denote by $P^{\delta}_{a,b}$ the law, on $C([0,1])$, of the $\delta$-dimensional Bessel bridge from $a$ to $b$ over the time interval $[0,1]$ (that is, the law of a $\delta$-dimensional Bessel process started at $a$ and conditioned to hit $b$ at time $1$). We shall denote by $E^{\delta}_{a,b}$ the expectation operator for $P^{\delta}_{a,b}$. Morever, when $a=b=0$, we shall drop the subindices and use the compact notations $P^\delta$ and $E^\delta$.
Note that, for all $a,b \geq 0$, $P^{\delta}_{a,b}$ is the image of $Q^{\delta}_{a^{2},b^{2}}$ under the map $\omega\mapsto\sqrt{\omega}$.
In particular, under the measure $P^\delta$, for all $r \in (0,1)$, $X_{r}$ admits the density $p^{\delta}_{r}$ on $\mathbb{R}_{+}$, where
by \eqref{one_pt_density_sqred_bridge_00} 
\begin{equation}
\label{one_pt_density_bridge_00}
 p^{\delta}_{r}(a) = 2 a \, q^{\delta}_{r}(a^{2})=
\frac{a^{\delta-1}}{2^{\frac\delta2-1}\,\Gamma(\frac{\delta}{2})(r(1-r))^{\delta/2}}\, \exp \left(- \frac{a^{2}}{2r(1-r)} \right), \quad a \geq 0 .
\end{equation}

\subsection{Pinned bridges}

Let $\delta>0$. For all $x \geq 0$ and $r \in (0,1)$, we denote by $Q^{\delta}_{0,0}  [\, \cdot \, | \, X_{r} = x]$ the law, on $C([0,1])$, of a $\delta$-dimensional squared Bessel bridge between $0$ and $0$, pinned at $x$ at time $r$ (that is, conditioned to hit $x$ at time $r$). Such a probability law can be constructed using the same procedure as for the construction of squared Bessel bridges. One similarly defines, for all $a \geq 0$ and $r \in (0,1)$, the law  $P^{\delta} [\ \cdot \ \, | \, X_{r} = a]$ of a $\delta$-dimensional Bessel bridge between $0$ and $0$ pinned at $a$ at time $r$. Note that the latter probability measure is the image of $Q^{\delta}_{0,0}  [ \ \cdot \ \, | \, X_{r} = a^{2}]$ under the map \eqref{sqrt_map}.

With these notations at hand, we now define a family of measures which will play an important role in the IbPF for Bessel bridges. Heuristically, they should be related to the local times of the solution $(u(t,x))_{t\geq 0, \, x \in [0,1]}$ to an SPDE having the law of a Bessel bridge as reversible measure. 

\begin{df}
For all $a \geq 0$ and $r\in(0,1)$, we set
\begin{equation}\label{Sigma}
\Sigma^\delta_r({\rm d}X \,|\, a) := \frac{p^{\delta}_{r}(a)}{a^{\delta-1}} \,
 P^{\delta} [ {\rm d} X \,| \, X_{r} = a],
\end{equation}
where $p^{\delta}_{r}$ is the probability density function of $X_{r}$ under $P^{\delta}:=P^{\delta}_{0,0}$, see \eqref{one_pt_density_bridge_00}. 
\end{df}
The measure $\Sigma^\delta_r(\,\cdot \,|\, a)$ is meant to be the \textit{Revuz measure} of the \textit{additive functional} corresponding to the diffusion local time of $(u(t,r))_{t\geq 0}$ at level $a \geq 0$ (see \cite[Chap. V]{fukushima2010dirichlet} and \cite[Chap. 6]{ma2012introduction} for this terminology). 

\begin{rk}
Note that, for all $r \in (0,1)$, by 
\eqref{one_pt_density_bridge_00}, we have
\[
 \frac{p^{\delta}_{r}(a)}{a^{\delta-1}} = 
\frac1{2^{\frac\delta2-1}\,\Gamma(\frac{\delta}{2})(r(1-r))^{\delta/2}}\, \exp \left(- \frac{a^{2}}{2r(1-r)} \right), \quad a > 0,
\]
and the right-hand side is well-defined also for $a=0$. It is this  quantity that we consider in equality \eqref{Sigma} above.
\end{rk}

To keep the formulae concise, for all $r \in (0,1)$ and $a \geq 0$, and all Borel function $\Phi : C([0,1]) \to \mathbb{R}_+$, we shall write with a slight abuse of language
\[ \Sigma^\delta_r(\Phi(X) \,|\, a) := \int \Phi(X) \ \Sigma^\delta_r({\rm d}X \,|\, a).
\]
In the sequel we will have to compute quantities of the form
\[
\Sigma^\delta_r\left(\exp(- \langle m,  X^2 \rangle) \,|\, a\right)
\]
for $m$ a finite Borel measure on $[0,1]$. In that perspective, we introduce some further notations. Given such a $m$, following the notation used in \cite{pitman1982decomposition} (see also Exercise (1.34), Chap. XI, of \cite{revuz2013continuous}), we denote by $\psi$ the function on $[0,1]$ given by 

\begin{equation}
\label{psi} 
\psi_r := \phi_r \int_{0}^{r} \phi_u^{-2} \d u = \phi_r \varrho_r, \qquad r \in [0,1],
\end{equation}
where $\varrho$ is as in \eqref{def_var_rho}.
Note that $\psi$ is the unique solution on $[0,1]$ of the Cauchy problem
\[\begin{cases}
\psi''({\rm d} r) = 2 \, \psi_{r} \, m({\rm d} r) \\
\psi_0=0, \quad \psi'_0= 1.
\end{cases} 
\]  
Moreover, we denote by $\hat{\psi}$ the function on $[0,1]$ given by
\begin{equation}
\label{psi_hat}
\hat{\psi}_r := \phi_1 \phi_r (\varrho_{1} - \varrho_r) = \psi_1 \phi_r - \psi_r \phi_1,\quad r \in [0,1].
\end{equation}
Note that $\hat{\psi}$ satisfies the following problem on $[0,1]$
\[\begin{cases}
\hat{\psi}''({\rm d} r) = 2 \, \hat{\psi}_{r} \, m({\rm d} r)  \\
\hat{\psi}_1=0, \quad \hat{\psi}'_1= -1.
\end{cases} 
\]
Note that the functions $\phi$, $\psi$ and $\hat{\psi}$ take positive values on $]0,1[$.
\begin{lm}\label{lap_cond_bridge} 
For all $r \in (0,1)$, $\delta>0$ and $a \geq 0$, the following holds:
\begin{equation}\label{bridge2}
\int \exp (- \langle m, X^{2} \rangle) \ \Sigma^\delta_r({\rm d}X \,|\, a) =
\frac1{2^{\frac\delta2-1}\,\Gamma(\frac{\delta}{2})} \,
\exp \left(-\frac{a^{2}}{2} C_r \right)  D_r^{\delta/2},
\end{equation}
where
\[
C_r = \frac{\psi_1}{\psi_r \hat{\psi}_r}, \qquad D_r = \frac{1}{\psi_r \hat{\psi}_r}. 
\]
\end{lm}

\begin{proof}
First note that by \eqref{relation_denisties_bes_besq} and \eqref{Sigma}, we have

\begin{equation}
\label{intermediate_expr} 
\begin{split}
\int \exp (- \langle m, X^{2} \rangle) \ \Sigma^\delta_r({\rm d}X \,|\, a) = 2 \, \frac{q^{\delta}_{r}(a^{2})}{a^{\delta-2}} \, Q^{\delta}_{0,0}  [\exp (- \langle m, X \rangle ) \, | \, X_{r} = a^{2}]. 
\end{split}
\end{equation}
To obtain the claim, it therefore suffices to compute 
\[ Q^{\delta}_{0,0}  [\exp (- \langle m, X \rangle ) \, | \, X_{r} = a^{2}]. \] 
Since $Q^{\delta}_{0,0} := Q^\delta_0 [ \, \cdot \, | X_1 = 0]$, one can rewrite the above expression as
\[Q^{\delta}_0 [\exp (- \langle m, X \rangle ) \, | \, X_{r} = a^{2}, X_1 = 0].\]
Therefore, \eqref{bridge2} follows from the computation of the Laplace transform of the conditional law 
$ Q^{\delta}_0$ given the value of the pair $(X_r,X_1)$.
To this aim, consider two Borel functions $f,g: \mathbb{R}_{+} \to \mathbb{R}_{+}$. We have
\begin{align*}
&\int_{0}^{\infty} \int_{0}^{\infty} Q^{\delta}_{0} [\exp (- \langle m, X \rangle) \, | \, X_{r} = x, X_{1}=y] \, q^{\delta}_{r}(a^{2},x) q^{\delta}_{1-r}(x,y) f(x) g(y) \d x \d y = \\
&= Q^{\delta}_{0} \left[\exp (- \langle m, X \rangle) f(X_{r}) g(X_{1}) \right] 
= \phi_1^{\frac\delta2} Q^{\delta}_{0} \left[ f \left( \phi_r^{2} X_{\varrho_r} \right) g \left( \phi_1^{2} X_{\varrho_1} \right) \right] = \\
&= \phi_1^{\delta/2-2} \phi_r^{-2} \int_{0}^{\infty} \int_{0}^{\infty} q^{\delta}_{\varrho_r} \left(0,\frac{x}{\phi_r^{2}}\right) q^{\delta}_{\varrho_1 - \varrho_r}\left(\frac{x}{\phi_r^{2}},\frac{y}{\phi_1^{2}}\right)f(x)g(y) \d x \d y  .
\end{align*}
Here, we used Lemma \ref{measure_change} to obtain the second equality. Since the functions $f$ and $g$ are arbitrary we deduce that:
\[ \begin{split} 
& Q^{\delta}_{0} [\exp (- \langle m, X \rangle) \, | \, X_{r} = x, X_{1}=y] \, =
\, \phi_1^{\delta/2-2} \phi_r^{-2} \frac{q^{\delta}_{\varrho_r} \left(0,\frac{x}{\phi_r^{2}}\right) q^{\delta}_{\varrho_1 - \varrho_r}\left(\frac{x}{\phi_r^{2}},\frac{y}{\phi_1^{2}}\right)}{q^{\delta}_{r}(0,x) \,q^{\delta}_{1-r}(x,y)} 
\end{split}\]
${\rm d} x \d y$ a.e. on ${\mathbb{R}_{+}^{*}}\times {\mathbb{R}_{+}^{*}}$. Since the family of measures $\left( Q^{\delta}_{x,y} \right)_{x,y \geq 0}$ is continuous in $(x,y) \in \mathbb{R}_{+}^{2}$ for the weak topology on probability measures (see \cite{revuz2013continuous}, Section XI.3), we deduce that, for all $x \geq 0$
\[\begin{split}
Q^{\delta}_{0,0}  [\exp (- \langle m, X \rangle) \, | \, X_{r} = x] &= \underset{\substack{y \to 0 \\y>0}}{\lim} \, \phi_1^{\delta/2-2} \phi_r^{-2} \frac{q^{\delta}_{\varrho_r} \left(0,\frac{x}{\phi_r^{2}}\right) q^{\delta}_{\varrho_1 - \varrho_r}\left(\frac{x}{\phi_r^{2}},\frac{y}{\phi_1^{2}}\right)}{q^{\delta}_{r}(0,x) \,q^{\delta}_{1-r}(x,y)}. 
\end{split}\]
But, by \eqref{density_besq_x_pos} and \eqref{density_besq_x_zero}, we have
\[\frac{q^{\delta}_{\varrho_r} \left(0,\frac{x}{\phi_r^{2}}\right) }{q^{\delta}_{r}(0,x)} = \left(\frac{r}{\varrho_{r}}\right)^{\frac\delta2} \phi_{r}^{2-\delta} \exp \left(-\frac{x}{2}\left(\frac{1}{\phi_{r}^{2} \varrho_{r}}- \frac{1}{r} \right) \right)
\]
and
\[ \underset{\substack{y \to 0 \\y>0}}{\lim} \, \frac{q^{\delta}_{\varrho_1 - \varrho_r}\left(\frac{x}{\phi_r^{2}},\frac{y}{\phi_1^{2}}\right)}{q^{\delta}_{1-r}(x,y)} = \left(\frac{1-r}{\varrho_{1} -\varrho_{r}}\right)^{\frac\delta2} \phi_{1}^{2-\delta} \exp \left(-\frac{x}{2}\left(\frac{1}{\phi_{r}^{2} (\varrho_{1}-\varrho_{r})}- \frac{1}{1-r} \right) \right). \]
We thus obtain
\begin{equation}
\label{equality_bridge}
\begin{split}
&Q^{\delta}_{0,0}  [\exp (- \langle m, X \rangle) \, | \, X_{r} = x] =\\
&= \phi_1^{-\delta/2} \phi_r^{-\delta} \left(\frac{r(1-r)}{\varrho_{r}(\varrho_{1} -\varrho_{r})}\right)^{\frac\delta2} \exp \left(-\frac{x}{2}\left(\frac{\varrho_{1}}{\phi_{r}^{2} \varrho_{r}(\varrho_{1}-\varrho_{r})}- \frac{1}{r(1-r)} \right) \right) = \\
&= \left(\frac{r(1-r)}{\psi_{r} \hat{\psi}_{r}}\right)^{\frac\delta2} \exp \left(-\frac{x}{2}\left(\frac{\psi_{1}}{\psi_{r} \hat{\psi}_{r}}- \frac{1}{r(1-r)} \right) \right), 
\end{split}
\end{equation}
where the second equality follows from the relations \eqref{psi}-\eqref{psi_hat} defining $\psi$ and $\hat{\psi}$. Applying this equality to $x=a^{2}$, and replacing in \eqref{intermediate_expr}, we obtain the claim.
\end{proof}

\begin{rk}
Along the proof of the above Proposition, for $\delta > 0$, $a \geq 0$, $r \in (0,1)$ and $m$ as above, we also obtained from equality \eqref{equality_bridge} the following, useful expression
\begin{equation}
\label{cond_bridge}
\begin{split}
& Q^{\delta}_{0,0}  \left[\exp (- \langle m, X \rangle) \, | \, X_{r} = a^2\right] = E^{\delta} [\exp (- \langle m, X^{2} \rangle) \, | \, X_{r} = a] \\
& = \exp \left(-\frac{a^{2}}{2} \left( \frac{\psi_1}{\psi_r \hat{\psi}_r} - \frac{1}{r(1-r)} \right)\right)  \left(\frac{r(1-r)}{\psi_r \hat{\psi}_r}\right)^{\delta/2}.
\end{split}
\end{equation}
\end{rk}

\section{Integration by parts formulae}
\label{sect_ibpf_exp_func}

Here and in the sequel, we denote by $\mathcal{S}$ the linear span of all functionals on $C([0,1])$ of the form
\begin{equation}
\label{exp_functional}
C([0,1])\ni X \mapsto \exp \left( - \langle m, X^{2} \rangle \right)\in \mathbb{R} 
\end{equation}
where $m$ is a finite Borel measure on $[0,1]$. The elements of $\mathcal{S}$ are the functionals for which we will derive our IbPFs wrt the laws of Bessel bridges. 

\subsection{The statement}

After recalling the definition \eqref{kappadelta} of $\kappa(\delta) = \frac{(\delta-3)(\delta-1)}{4}$, for $\delta\in\R$, we can now state one of the main results of this article.

\begin{thm}
\label{statement_ibpf}
Let $\delta \in (0,\infty) \setminus \{1,3\}$, and set $k:=\lfloor \frac{3-\delta}{2} \rfloor \leq 1$. Then, for all $\Phi \in \mathcal{S}$ and $h \in C^2_c(0,1)$
\begin{equation}
\label{exp_fst_part_ibpf_a_b}
\begin{split}
& E^{\delta} (\partial_{h} \Phi (X) ) + E^{\delta} (\langle h '' , X \rangle \, \Phi(X) ) = \\
&=-\kappa(\delta)\int_{0}^{1}  
 h_{r}  \int_0^\infty a^{\delta-4} \Big[ \mathcal{T}^{\,2k}_{a} \, \Sigma^\delta_r(\Phi (X) \,|\, \cdot\,) \Big]
  \d a \d r,
\end{split}
\end{equation}
where $\mathcal{T}^{\,n}_{x}$ is the Taylor remainder defined in \eqref{eq:taylor}.
On the other hand, when $\delta \in \{1,3\}$, the following formulae hold for all $\Phi \in \mathcal{S}$ and $h \in C^2_c(0,1)$
\begin{equation}
\label{exp_fst_part_ibpf_a_b_3}
E^{3}(\partial_{h} \Phi (X) ) + E^{3}(\langle h '' , X \rangle \, \Phi(X) ) = 
-\frac{1}{2} \int_{0}^{1} h_{r} \, \Sigma^3_r(\Phi (X) \,|\, 0) \d r, 
\end{equation}
\begin{equation}
\label{exp_fst_part_ibpf_a_b_1}
\begin{split}
E^{1} (\partial_{h} \Phi (X) ) + E^{1} (\langle h '' , X \rangle \, \Phi(X) ) = 
\frac{1}{4} \int_{0}^{1} h_{r} \, \frac{{\rm d}^{2}}{{\rm d} a^{2}} \, \Sigma^1_r (\Phi(X) \, | \, a) \, \biggr\rvert_{a=0}  \d r. 
\end{split}
\end{equation}
\end{thm}

\begin{rk}
Note that the last integral in \eqref{exp_fst_part_ibpf_a_b} is indeed convergent. Indeed, by Lemma \ref{lap_cond_bridge} $\mathcal{T}^{\,2k}_{a} \, \Sigma^\delta_r(\Phi (X) \,|\, \cdot\,)$ is the Taylor remainder of order $2k$ at $0$ of a smooth, even, function, see \eqref{eq:taylor} above. Hence, near $0$, the integrand is of order $O(a^{\delta+ 2k - 2})$. Since, $\delta + 2k -2 > -1$, the integral is convergent at $0$. On the other hand, near $\infty$, the integrand is of order $O(a^{\delta + 2k -4 })$. Since $\delta + 2k -4 < -1$, integrability also holds at $+ \infty$.
\end{rk}

\begin{rk}
For all $\delta \in (1,3)$ the right-hand side in the IbPF \eqref{exp_fst_part_ibpf_a_b} takes the form
\[ -\kappa(\delta)\int_{0}^{1} h_{r}  \int_0^\infty a^{\delta-4} \Big[ \Sigma^\delta_r(\Phi (X) \,|\, a) - \Sigma^\delta_r(\Phi (X) \,|\, 0) \Big]  \d a \d r. \]
Note that, while there is a transition in the structure of the IbPF at the values $\delta=3$ and $\delta=1$, with the order of the Taylor series changing at these critical values, no such transition occurs at $\delta=2$. This might seem surprising given the transition that the Bessel bridges undergo at $\delta=2$, which is the smallest value of $\delta$ satisfying
\[ P^{\delta} \left[ \exists r \in \,]0,1[ \ : \, X_{r} = 0 \right] = 0. \]
This lack of transition at $\delta=2$ is related to the fact that, as a consequence of Lemma \ref{lap_cond_bridge}, we have for all $\Phi\in{\mathcal E}$:
\[ \frac{\rm d}{{\rm d}a} \, \Sigma^\delta_r(\Phi (X) \,|\, a) \biggr\rvert_{a=0} = 0. \]
\end{rk}

\begin{rk}
In the IbPF \eqref{exp_fst_part_ibpf_a_b}, the last term may equivalently be written as
\begin{equation}
\label{last_term}
 -\kappa(\delta)\int_{0}^{1}  
 h_{r}  \int_0^\infty a^{-3} \Big[ \mathcal{T}^{\,2k}_{a} \, \Sigma^\delta_r(\Phi (X) \,|\, \cdot\,) \Big]
  m_{\delta}({\rm d}a) \d r 
\end{equation}
where $m_{\delta}$ is the measure on $\mathbb{R}_{+}$ defined by 
\[ m_{\delta}({\rm d} a) = \mathbf{1}_{a>0} \, a^{\delta-1} \d a. \]
Note that $m_{\delta}$ is a reversible measure for the $\delta$-dimensional Bessel process. Actually, if $(X_{t})_{t \geq 0}$ is a $\delta$ dimensional Bessel process, we can construct a bicontinuous family of \textit{diffusion local times} $\left(\ell^{a}_{t}\right)_{a, t \geq 0}$, satisfying the occupation times formula
\[  
\int_{0}^{t} f \left( X_{s} \right) {\rm d} s = \int_{0}^{+\infty}  f(a) \, \ell^{a}_{t} \, m_{\delta}({\rm d}a),
\]     
for all $f: \mathbb{R}_{+} \to \mathbb{R}_{+}$ bounded and Borel. We hope that such a property should hold also for $(u(t,x))_{t \geq 0}$, for all $x \in (0,1)$ where $u$ is the hypothetical solution of the dynamics corresponding to $P^{\delta}$. In that case the term \eqref{last_term} should correspond, in the dynamics, to a drift in $u^{-3}$ integrated against renormalised local times. We shall develop this idea more in detail in Section \ref{sect_conj_dynamics} below. 
\end{rk}

\subsection{Proof of Theorem \ref{statement_ibpf}}

We first state a differential relation satisfied by the product of the functions $\psi$ and $\hat{\psi}$ associated as above with a finite Borel measure $m$ on $[0,1]$. This relation is the skeleton of all the IbPFs for $P^{\delta}$, $\delta >0$ : the latter will all be deduced from the former with a simple multiplication by a constant (depending on the parameter $\delta$).

\begin{lm}\label{thm}
Let $m$ be a finite Borel measure on $[0,1]$, and consider the functions $\psi$ and $\hat{\psi}$ as in \eqref{psi} and \eqref{psi_hat}. Then, for all $h \in C^{2}_{c}(0,1)$ and $\delta > 0$, the following equality holds
\begin{equation}\label{laplace}
\int_{0}^{1} \sqrt{\psi_r \hat{\psi}_r} \left( h''_r \d r - 2 h_r \, m({\rm d} r) \right)
 = - \frac{1}{4} \psi_1^{2} \int_{0}^{1} h_r (\psi_r \hat{\psi}_r)^{-\frac{3}{2}} \d r.
\end{equation}   
\end{lm}

\begin{proof}
Performing an integration by parts, we can rewrite the left-hand side as
\[ \int_{0}^{1} h_r \left( \frac{\mathrm{d}^{2}}{\mathrm{d} r^{2}} - 2 \, m({\rm d} r) \right) \left( \psi_r \hat{\psi}_r \right)^{\frac{1}{2}}. \]
Note that here we are integrating wrt the signed measure
\[ \left( \frac{\mathrm{d}^{2}}{\mathrm{d} r^{2}} - 2 \, m({\rm d} r) \right) \left( \psi_r \hat{\psi}_r \right)^{\frac{1}{2}} = \frac{\mathrm{d}^{2}}{\mathrm{d} r^{2}} \left( \psi_r \hat{\psi}_r \right)^{\frac{1}{2}} - 2 \left( \psi_r \hat{\psi}_r \right)^{\frac{1}{2}} \, m({\rm d} r).\]
Now, we have
\begin{align*} 
\frac{\mathrm{d}^{2}}{\mathrm{d} r^{2}}  \left( \psi \hat{\psi} \right)^{\frac{1}{2}}  = \frac{1}{2} \frac{ \psi''\hat{\psi} + 2 \psi'\hat{\psi}' + \psi \hat{\psi}''}{(\psi\hat{\psi})^{\frac12}} - \frac{1}{4} \frac{(\psi'\hat{\psi} + \psi \hat{\psi}')^{2}}{(\psi\hat{\psi})^{3/2}}.
\end{align*}
Recalling that $\psi''= 2 \psi \, m$ and $\hat{\psi}''=2 \hat{\psi} \, m$, we obtain
\begin{align*} 
\left( \frac{\mathrm{d}^{2}}{\mathrm{d} r^{2}} - 2 \, m ({\rm d} r) \right) \left( \psi \hat{\psi} \right)^{\frac{1}{2}} = & \frac{\psi'\hat{\psi}'\psi \hat{\psi} - \frac{1}{4} (\psi'\hat{\psi} + \psi \hat{\psi}')^{2}}{(\psi\hat{\psi})^{3/2}} \\
= & -\frac{1}{4} \frac{(\psi'\hat{\psi} - \psi \hat{\psi}')^{2}}{(\psi\hat{\psi})^{3/2}}.
\end{align*}
Using the expressions \eqref{psi} and \eqref{psi_hat} for $\psi$ and $\hat{\psi}$, we easily see that 
\[
\psi'_r\hat{\psi}_r -\psi \hat{\psi}'_r = \psi_1, \qquad r \in (0,1).
\]
Hence, we obtain the following equality of signed measures:
\[ \left( \frac{\mathrm{d}^{2}}{\mathrm{d} r^{2}} - 2 \, m \right)  \left( \psi \hat{\psi} \right)^{\frac{1}{2}} = - \frac{1}{4} \frac{\psi_1^{2}}{(\psi_r\hat{\psi}_r)^{3/2}} \d r. \]
Consequently, the left-hand side in \eqref{laplace} is equal to 
\[ - \frac{1}{4} \psi_1^{2} \int_{0}^{1} \d r \ h_r  \left(\psi_r\hat{\psi}_r\right)^{-3/2}. \]
The claim follows.
\end{proof}

As a consequence, we obtain the following preliminary result.
\begin{lm}
Let $m$ be a finite measure on $[0,1]$, and let $\Phi:C([0,1]) \to \mathbb{R}$ be the functional thereto associated  as in \eqref{exp_functional}. Then, for all $\delta>0$ and $h \in C^{2}_{c}(0,1)$,
\begin{equation}\label{exp_fst_part_ibpf0}
\begin{split}
& E^{\delta} (\partial_{h} \Phi (X) ) + E^{\delta} (\langle h '' , X \rangle \, \Phi(X) ) =
\\ & =-\frac{\Gamma(\frac{\delta+1}{2})}{2^{\frac32}\,\Gamma(\frac{\delta}{2})} \, \psi_1^{-\frac{\delta-3}{2}}\int_{0}^{1} 
h_r \left(\psi_r\hat{\psi}_r\right)^{-\frac32} \d r,
\end{split}
\end{equation}
where $\psi$ and $\hat{\psi}$ are associated with $m$ as in \eqref{psi} and \eqref{psi_hat}. 
\end{lm}

\begin{proof}
By the expression \eqref{exp_functional} for $\Phi$, we have
\[ \partial_{h} \Phi (X) = - 2 \langle X h , m \rangle \, \Phi(X).\]
Therefore
\begin{align*}
&E^{\delta} (\partial_{h} \Phi (X) ) + E^{\delta} (\langle h '' , X \rangle \, \Phi(X) ) =Q^{\delta}_{0,0}  \left[ \left( \langle h '',\sqrt{X} \rangle - 2 \langle h \sqrt{X},  m\rangle \right) \, e^{- \langle m, X \rangle} \right] =
\\ &=\int_{0}^{1} ( h''_{r} \d r  - 2 h_{r} \, m({\rm d} r) )
\int_{0}^{+ \infty} \! \Gamma\left(\frac{\delta}2,\frac1{2r(1-r)}\right)({\rm d}a) \sqrt{a} \, Q^{\delta}_{0,0}  \left[\left.e^{- \langle m, X \rangle} \, \right| \, X_{r} = a \right].
\end{align*}
By \eqref{cond_bridge} we obtain:
\begin{align*}
&E^{\delta} (\partial_{h} \Phi (X) ) + E^{\delta} (\langle h '', X \rangle \, \Phi(X))=
\\ &=  \int_{0}^{1} \, \left(h''_{r} \d r  - 2 h_{r} \, m({\rm d} r) \right) 
\frac{\Gamma(\frac{\delta+1}2)}{\Gamma(\frac{\delta}2)}\left(\frac{C_r}2\,\psi_1^\delta\right)^{-\frac12}
\int_{0}^{+ \infty}  \, \Gamma\left(\frac{\delta+1}2,\frac{C_r}2\right)({\rm d}a) 
\\ &= \sqrt{2}\,\frac{\Gamma(\frac{\delta+1}2)}{\Gamma(\frac{\delta}2)} \psi_1^{-\frac{\delta+1}2}\int_{0}^{1} {\rm d} r \, \left(h''_{r} \d r - 2 h_{r} \, m({\rm d} r) \right) \sqrt{\psi_r\hat\psi_r} .
\end{align*}
Finally, by \eqref{laplace}, the latter expression is equal to
\[
-\frac{\Gamma(\frac{\delta+1}{2})}{2^{\frac32}\,\Gamma(\frac{\delta}{2})} \, \psi_1^{-\frac{\delta-3}{2}}\int_{0}^{1} 
h_r \left(\psi_r\hat{\psi}_r\right)^{-\frac32} \d r
\]
and the proof is complete.
\end{proof}

Apart from the above lemma, the proof of the IbPF for $P^{\delta}$, $\delta > 0$, will require integral expressions for negative Gamma values. 
For all $x\in \mathbb{R}$ we set $\lfloor x \rfloor:=\sup\{k\in \mathbb{Z}: k \leq x\}$. We also use the notation $\Z^-:=\{n\in\Z: n\leq 0\}$.
\begin{lm}
\label{neg_gamma}
For all $x\in \mathbb{R}\setminus\Z^-$
\[ \Gamma (x) = \int_{0}^{\infty} t^{x-1} \mathcal{T}^{\,\lfloor - x \rfloor}_{t} (e^{- \, \cdot \,}) \d t.
\]
\end{lm}

\begin{proof}
By Remark \ref{laplace_mu_neg} we have
\[
\int_{0}^{\infty} t^{x-1} \mathcal{T}^{\,\lfloor - x \rfloor}_{t} (e^{- \, \cdot \,}) \d t
=\Gamma(x) \,\langle \mu_{\alpha} , e^{-\cdot } \rangle=\Gamma(x) \,1^{x}=\Gamma(x),
\] 
and the claim  follows.  
\end{proof}
From Lemma \ref{neg_gamma} we obtain for all $C>0$, $x\in \mathbb{R}\setminus\Z^-$
\begin{equation}\label{forallx}
\Gamma(x) \, C^{-x} = 2^{1-x}\int_0^{+\infty} a^{2x-1} \left( e^{-C\frac{a^2}2} - \sum_{0\leq j\leq \lfloor -x \rfloor}  \frac{(-C)^{j} a^{2j}}{2^jj !} \right) \d a
\end{equation}
by a simple change of variable $t=Cb^2/2$.
Then \eqref{forallx} can be rewritten as follows
\begin{equation}\label{forallx'}
\Gamma(x) \, C^{-x} = 2^{1-x}\int_0^{+\infty} a^{2x-1} \, \mathcal{T}^{\, 2\lfloor - x \rfloor}_{a}  \left( e^{-C\frac{(\cdot)^2}2}  \right) \d a,
\quad x\in \mathbb{R}\setminus\Z^-.
\end{equation}
We can finally prove the main statement of this section.
\begin{proof}[Proof of Theorem \ref{statement_ibpf}]
Let first $\delta>0$ and $\delta\notin\{1,3\}$. 
Then by \eqref{exp_fst_part_ibpf0}
\[
\begin{split}
& E^{\delta} (\partial_{h} \Phi (X) ) + E^{\delta} (\langle h '' , X \rangle \, \Phi(X) ) =
\\ &= - \frac{\Gamma(\frac{\delta+1}{2})}{2^{3/2} \,\Gamma(\frac{\delta}{2})} \, \int_{0}^{1}  h_{r}\, \left(\frac{\psi_1}{\psi_r\hat\psi_r}\right)^{\frac{3-\delta}{2}}  \,  \left(\psi_r\hat{\psi}_r\right)^{-\frac{\delta}{2}}\d r 
 \\ &= - \frac{\Gamma(\frac{\delta+1}{2})}{2^{3/2} \,\Gamma(\frac{\delta}{2})} \, \int_{0}^{1}  h_{r}\,  C_{r}^{\frac{3-\delta}{2}}  \, 
D_{r}^{\delta/2}\d r 
  \\ & = - \frac{\Gamma(\frac{\delta+1}{2})}{2^{3/2} \Gamma(\frac{\delta}{2})} \, \frac{ 2^{\frac{5-\delta}{2}}}{\Gamma\left(\frac{\delta-3}{2}\right)} \int_{0}^{1} h_{r} \,D_{r}^{\delta/2}
 \int_0^\infty a^{\delta-4} \, \mathcal{T}^{\,2k}_{a} e^{-\frac{C_r}{2} (\cdot)^{2}}\d  a \d r,
\end{split} \] 
where we used \eqref{forallx'} with $C=C_{r}$ and $x = \frac{\delta-3}{2}$ to obtain the last line. Recalling the expression \eqref{bridge2} for $\Sigma^\delta_r(\Phi(X)\,|\,a)$, we thus obtain 
\[\begin{split}
& E^{\delta} (\partial_{h} \Phi (X) ) + E^{\delta} (\langle h '' , X \rangle \, \Phi(X) ) =
\\& = - \frac{\Gamma(\frac{\delta+1}{2})}{\Gamma(\frac{\delta-3}{2})} \int_{0}^{1} h_{r} 
 \int_0^\infty a^{\delta-4} \, \mathcal{T}^{\,k}_{2a} \, \Sigma^\delta_r(\Phi(X)\,|\,a) \,\d  a \d r.
\end{split}
\]
Now, since $\delta\notin\{1,3\}$,
\[
\textstyle{\Gamma(\frac{\delta+1}{2}) = \frac{\delta-1}{2}\, \Gamma(\frac{\delta-1}{2}) =
\frac{\delta-1}{2}\, \frac{\delta-3}{2}\, \Gamma(\frac{\delta-3}{2}) = \kappa(\delta)\, \Gamma(\frac{\delta-3}{2}).
}
\]
Therefore $\frac{\Gamma(\frac{\delta+1}{2})}{\Gamma(\frac{\delta-3}{2})}=\kappa(\delta)$ and we obtain the claim. 

There remains to treat the critical cases $\delta \in \{1,3\}$. 
By linearity, we may assume that $\Phi$ is of the form \eqref{exp_functional}.
For $\delta=3$ we have by \eqref{exp_fst_part_ibpf0}
\[
\begin{split}
& E^{3} (\partial_{h} \Phi (X) ) + E^{3} (\langle h '' , X \rangle \, \Phi(X) ) 
=  -\frac{1}{2^{\frac32}\,\Gamma(\frac3{2})} \int_{0}^{1} h_{r} \left(\psi_r\hat{\psi}_r\right)^{-\frac32}\d r.
\end{split}
\]
By \eqref{bridge2} this equals
\[ - \frac{1}{2} \int_{0}^{1} \d r \, h_r \, \Sigma^3_r(\Phi(X) \,|\, 0)
\]
and the proof is complete.
For $\delta=1$, by \eqref{exp_fst_part_ibpf0}, we have
\[E^{1} (\partial_{h} \Phi (X) + \langle h '' , X \rangle \, \Phi(X) )  = -\frac1{2\sqrt{2\pi}} \, \psi_1\int_{0}^{1} 
h_r \left(\psi_r\hat{\psi}_r\right)^{-\frac32} \d r. \]
But by \eqref{bridge2} we have, for all $r \in (0,1)$
\[
 \frac{{\rm d}^{2}}{{\rm d} a^{2}} \, \Sigma^1_r (\Phi(X) \, | \, a)\,  \biggr\rvert_{a=0} = -
\frac{C_r \, D_r^{\frac12}}{2^{-\frac12}\,\Gamma(\frac{1}{2})} = -\sqrt{\frac{2}{\pi}}\,  \psi_1 \left(\psi_r\hat{\psi}_r \right)^{-\frac{3}{2}}.
\]
The claimed IbPF follows.
\end{proof}

\begin{rk}\label{weaker}
In \cite{zambotti2005integration} for the reflecting Brownian motion, and then in \cite{grothaus2016integration} for the Reflecting Brownian bridge,
a different formula was proved in the case $\delta=1$. In our present notations, for $(\beta_r)_{r\in[0,1]}$ a Brownian bridge and $X:=|\beta|$,
the formula reads 
\begin{equation}\label{|beta|}
\E(\partial_h\Phi(X)) + \E(\langle h '' , X \rangle \, \Phi(X) )= \lim_{\epsilon\to 0}2 \, \E\left(\Phi(X)\int_0^1 h_r\left[ \left(\dot{\beta^\epsilon_r}\right)^2
-c^\epsilon_r\right] {\rm d} L^0_r \right),
\end{equation}
where $\Phi: H\to\R$ is any Lipschitz function, $h\in C^2_0(0,1)$, $L^0$ is the standard local time of $\beta$ at $0$ and for some even smooth mollifier $\rho_\epsilon$ we set
\[
\beta^\epsilon:=\rho_\epsilon*\beta, \qquad c^\epsilon_r:=\frac{\|\rho\|_{L^{2}(0,1)}^{2}}{\epsilon}.
\]
The reason why \eqref{|beta|} is strictly weaker than \eqref{exp_fst_part_ibpf_a_b_1}, is that the former depends explicitly on $\beta$, while the 
latter is written only in terms of $X$. This will become crucial when we compute the SPDE satisfied by $u$ for $\delta=1$ in Theorem \ref{fukushima_decomposition} below. 
\end{rk}

As a consequence of Theorem \ref{statement_ibpf}, we retrieve the following known results, see Chapter 6 of \cite{zambotti2017random} and \eqref{ibpf_larger_three}-\eqref{ibpf_three} above.
\begin{prop}
\label{already_known_ibpf0}
Let $\Phi \in \mathcal{S}$ and $h \in C^{2}_{c}(0,1)$. Then, for all $\delta > 3$, the following IbPF holds
\[
 E^{\delta}(\partial_{h} \Phi (X) ) + E^{\delta}(\langle h '' , X \rangle \, \Phi(X) ) = - \kappa(\delta) \, E^{\delta} (\langle h , X^{-3} \rangle \, \Phi(X) ). 
\]
Moreover, for $\delta = 3$, the following IbPF holds
\[
\begin{split}
& E^{3}(\partial_{h} \Phi (X) )+ E^{3}(\langle h '' , X \rangle \, \Phi(X) ) =
\\ & = - \int_{0}^{1} \d r \,  \frac{h_r}{\sqrt{2 \pi r^{3} (1-r)^{3}}} \, E^{3} [\Phi(X) \, | \, X_{r} = 0].
\end{split} 
\]
\end{prop}

\begin{proof}
 For $\delta>3$ we have $k:=\lfloor \frac{3-\delta}{2} \rfloor < 0$, and by \eqref{exp_fst_part_ibpf_a_b} 
  \[
 \begin{split}
& E^{\delta} (\partial_{h} \Phi (X) ) + E^{\delta} (\langle h '' , X \rangle \, \Phi(X) )= 
 \\ & = -\kappa(\delta)\int_{0}^{1}  
 h_{r}  \int_0^\infty a^{\delta-4} \, \Sigma^\delta_r(\Phi (X)\,|\, a) 
  \d a \d r 
  \\ & = -\kappa(\delta)\int_{0}^{1}  
 h_{r}  \int_0^\infty a^{-3} \, p^{\delta}_{r}(a)\, E^{\delta}[\Phi(X) \, | \, X_{r} = a]
  \d a \d r 
  \\ & = - \kappa(\delta) \, E^{\delta}(\langle h , X^{-3} \rangle \, \Phi(X) ). 
\end{split}
\]
For $\delta=3$, it suffices to note that, for all $r \in (0,1)$
\[ \frac{1}{2} \, \lim_{\epsilon \downarrow0} \,  \frac{p^{3}_{r}(\epsilon)}{\epsilon^{2}} = \frac{1}{\sqrt{2 \pi r^{3} (1-r)^{3}}},  \]
so that
\[\frac{1}{2} \, \Sigma^3_r(\Phi (X) \,|\, 0\,)  = \frac{1}{\sqrt{2 \pi r^{3} (1-r)^{3}}} E^{3}[\Phi(X) \, | \, X_{r} = 0], \]
and the proof is complete thanks to \eqref{exp_fst_part_ibpf_a_b_3}.
\end{proof}

\section{The dynamics via Dirichlet forms}
\label{sect_Dirichlet}

In this section we exploit the IbPF obtained above to construct a weak version of the gradient dynamics associated with $P^{1}$, using the theory of Dirichlet forms. The reason for considering the particular value $\delta=1$ is that we can exploit a representation of the Bessel bridge in terms of a Brownian bridge, for which the corresponding gradient dynamics is well-known and corresponds to a linear stochastic heat equation. This representation was already used in \cite{vosshallthesis} which constructed a quasi-regular Dirichlet form associated with $P^1$, a construction which does not follow from the IbPF \eqref{exp_fst_part_ibpf_a_b_1} due to the distributional character of its last term. Using this construction, we exploit the IbPF \eqref{exp_fst_part_ibpf_a_b_1} to prove that the associated Markov process, at equilibrium, satisfies \eqref{formal1}. The treatment of the particular value $\delta=1$ is also motivated by potential applications to scaling limits of dynamical critical pinning models, see e.g. \cite{vosshallthesis} and \cite{deuschel2018scaling}.

For the sake of our analysis, instead of working on the Banach space $C([0,1])$, it shall actually be more convenient to work on the Hilbert space $H:=L^{2}(0,1)$ endowed with the $L^2$ inner product
\[ \langle f, g \rangle = \int_0^1 f_r\, g_r \d r, \quad f,g \in H. \]
We shall denote by $\| \cdot \|$ the corresponding norm on $H$. Moreover we denote by $\mu$ the law of $\beta$ on $H$, where $\beta$ is a Brownian bridge from $0$ to $0$ over the interval $[0,1]$. We shall use the shorthand notation $L^{2} (\mu)$ for the space $L^2(H,\mu)$.

\subsection{The one-dimensional random string}

Consider the Ornstein-Uhlenbeck semigroup $(\mathbf{Q}_{t})_{t \geq 0}$ on $H$ defined, for all $F \in L^{2} (\mu)$ and $z \in H$, by
\[ \mathbf{Q}_{t} F (z) := \mathbb{E} \left[ F(v_{t}(z)) \right], \quad t \geq 0,
\]
where $(v_{t}(z))_{t \geq 0}$ is the solution to the stochastic heat equation on $[0,1]$ with initial condition $z$, and with homogeneous Dirichlet boundary conditions
\begin{equation}
\label{solution_she}
 \begin{split}
\begin{cases}
\frac{\partial v}{\partial t} = \frac{1}{2} \frac{\partial^{2} v}{\partial x^{2}} + \xi \\
v(0,x) = z(x), \qquad & x\in[0,1] \\
v(t,0)= v(t,1)= 0, \qquad & t > 0  
\end{cases}
\end{split}
\end{equation}
with $\xi$ a space-time white noise on $\mathbb{R}_{+} \times [0,1]$. Recall that $v$ can be written explicitly in terms of the fundamental solution $(g_{t}(x,x'))_{t \geq 0, \, x,x' \in (0,1)}$ of the stochastic heat equation with homogeneous Dirichlet boundary conditions on $[0,1]$, which by definition is the unique solution to
\[ \begin{split}
\begin{cases}
\frac{\partial g}{\partial t} = \frac{1}{2} \frac{\partial^{2} g}{\partial x^{2}} \\
g_{0}(x,x') = \delta_{x}(x') \\
g_{t}(x,0)= g_{t}(x,1)= 0.  
\end{cases}
\end{split}\]
Recall further that $g$ can be represented as follows:
\[
\forall t >0, \quad \forall x, x' \geq 0, \quad g_{t}(x,x') = \sum_{k=1}^{\infty} e^{-\frac{\lambda_{k}}{2}t} e_{k}(x) e_{k}(x'), 
\]
where $(e_{k})_{k \geq 1}$ is the complete orthornormal system of $H$ given by
\[e_{k}(x) := \sqrt{2} \sin(k \pi x), \quad x \in [0,1], \quad k \geq 1\] 
and $\lambda_{k} := k^{2} \pi^{2}$, $k \geq 1$.
We can then represent $u$ as follows:
\begin{equation}
\label{expr_solution_she}
v(t,x) = z(t,x) + \int_{0}^{t} \int_{0}^{1} g_{t-s}(x,x') \, \xi ({\rm d} s, \d x'), 
\end{equation}
where 
\begin{equation}
\label{expr_solution_he}
z(t,x) := \int_{0}^{1} g_{t}(x,x') z(x') \d x',
\end{equation}
and where the double integral is a stochastic convolution. In particular, it follows from this formula that $v$ is a Gaussian process. An important role will be played by its covariance function. Namely, for all $t \geq 0$ and $x,x' \in (0,1)$, we set
\[ q_{t}(x,x') := \text{Cov}(v(t,x) , v(t,x')) = \int_{0}^{t} g_{2 \tau}(x,x') \d \tau. \]
We also set
\[  q_{\infty}(x,x') := \int_{0}^{\infty} g_{2 \tau}(x,x') \d \tau=\mathbb{E}[\beta_x\beta_{x'}]
= x \wedge x' - x x' . \]
For all $t \geq 0$, we set moreover
\[ q^{t}(x,x') := q_{\infty}(x,x') - q_{t}(x,x') = \int_{t}^{\infty} g_{2 \tau}(x,x') \d \tau.\]
When $x=x'$, we will use the shorthand notations $q_{t}(x), q_{\infty}(x)$ and $q^{t}(x)$ instead of $q_{t}(x,x), q_{\infty}(x,x)$ and $q^{t}(x,x)$ respectively. Finally, we denote by $(\Lambda,D(\Lambda))$ the Dirichlet form associated with $(\mathbf{Q}_{t})_{t \geq 0}$ in $L^{2} (H,\mu)$, and which is given by
\[ \Lambda(F,G) = \frac{1}{2} \int_{H} \langle \nabla F, \nabla G \rangle \d \mu, \quad F,G \in D(\Lambda) = W^{1,2}(\mu), 
\] 
where we recall that $\mu$ denotes the law of a standard Brownian bridge on $[0,1]$. Here, for all $F \in W^{1,2}(\mu)$, $\nabla F: H \to H$ is the gradient of $F$, see \cite{dpz3}. The corresponding family of resolvents $(\mathbf{R}_\lambda)_{\lambda>0}$ is then given by
\[ \mathbf{R}_\lambda F (z) = \int_0^\infty e^{-\lambda t} \mathbf{Q}_{t} F(z) \d t, \quad z \in H, \, \lambda >0, \qquad F \in L^{2}(\mu).
\]

\subsection{Dirichlet form}

In this section we introduce the Dirichlet form associated with our equation \eqref{spde=1} and the associated Markov process $(u_t)_{t\geq 0}$. We stress that these objects were already constructed in \cite[Chap. 5]{vosshallthesis}.

Let $\mathcal{F} \mathcal{C}^{\infty}_{b}(H)$ denote the space of all functionals $F:H \to \mathbb{R}$ of the form
\begin{equation}\label{Fexp} 
F (z) = \psi(\langle l_{1}, z \rangle, \ldots, \langle l_{m}, z \rangle ), \quad z \in H, 
\end{equation} 
with $m \in \mathbb{N}$, $\psi \in C^{\infty}_{b}(\mathbb{R}^{m})$, and $l_{1}, \ldots, l_{m} \in \text{Span} \{ e_{k}, k \geq 1 \}$. Since Bessel bridges are \textit{nonnegative} processes, we are led to also introduce the closed subset $K \subset H$ of nonnegative functions
\[K:= \{ z \in H, \, \, z \geq 0 \, \, \text{a.e.} \}. \] 
Note that $K$ is a Polish space. We also define:
\[\mathcal{F} \mathcal{C}^{\infty}_{b}(K) := \left\{ F \big \rvert_{K} \, , \ F \in \mathcal{F} \mathcal{C}^{\infty}_{b}(H) \right\}. 
\]
Moreover, for $f \in \mathcal{F} \mathcal{C}^{\infty}_{b}(K)$ of the form $f=F \big \rvert_{K}$, with $F \in \mathcal{F} \mathcal{C}^{\infty}_{b}(H)$, we define $\nabla f : K \to H$ by
\[ \nabla f (z) = \nabla F(z), \quad z \in K, \]
where this definition does not depend on the choice of $F\in \mathcal{F} \mathcal{C}^{\infty}_{b}(H)$ such that $f=F\big \rvert_{K}$. We further denote by $\nu$ the law, on $K$, of the $1$-Bessel bridge from $0$ to $0$ on $[0,1]$ (so that $P^1$ is then the restriction of $\nu$ to $C([0,1])$). We shall use the shorthand $L^2(\nu)$ to denote the space $L^2(K,\nu)$. Denoting by $j : H \to K$ the absolute value map
\begin{equation}
\label{absolute_value_map}
j(z) := |z|, \quad z \in H,
\end{equation}
we remark that the map $L^{2}(\nu)\ni\varphi\mapsto \varphi \circ j\in L^{2}(\mu)$
is an isometry. Let us finally denote by $\mathcal{E}$ the bilinear form defined on $\mathcal{F} \mathcal{C}^{\infty}_{b}(K)$ by
\[ \mathcal{E}(f,g) := \frac{1}{2} \int_{K} \langle \nabla f , \nabla g \rangle \d \nu, \qquad f,g \in \mathcal{F} \mathcal{C}^{\infty}_{b}(K). 
\]

\begin{prop}
\label{closability}
The form $(\mathcal{E},\mathcal{F} \mathcal{C}^{\infty}_{b}(K))$ is closable. Its closure $(\mathcal{E},D (\mathcal{E}))$ is a local, quasi-regular Dirichlet form on $L^{2}(\nu)$. In addition, for all $f \in D (\mathcal{E})$, $f \circ j \in D(\Lambda)$, and we have
\begin{equation}
\label{isometry}
\forall f,g \in D(\mathcal{E}), \quad \mathcal{E}(f,g) = \Lambda(f \circ j, g \circ j). 
\end{equation}
\end{prop}
The proof of Proposition \ref{closability} is postponed to Appendix \ref{Proofs}.

\smallskip
Let $(Q_t)_{t \geq 0}$ be the contraction semigroup on $L^{2}(K,\nu)$ associated with the Dirichlet form $(\mathcal{E}, D(\mathcal{E}))$, and let $(R_\lambda)_{\lambda>0} $ be the associated family of resolvents. Let also $\mathcal{B}_{b}(K)$ denote the set of Borel and bounded functions on $K$. As a consequence of Prop. \ref{closability}, in virtue of Thm IV.3.5 and Thm V.1.5 in \cite{ma2012introduction}, we obtain the following result.

\begin{cor}
There exists a diffusion process $M=\{\Omega, \mathcal{F}, (u_t)_{t \geq 0}, (\mathbb{P}_x)_{x \in K} \}$ properly associated to $(\mathcal{E},D(\mathcal{E}))$, i.e. for all $\varphi \in L^{2}(\nu) \cap \mathcal{B}_b(K)$, and for all $t > 0$,  $E_{x}(\varphi(u_{t})), \, x \in K,$ defines an $\mathcal{E}$ quasi-continuous version of $Q_{t} \varphi$. Moreover, the process $M$ admits the following continuity property
\[\mathbb{P}_{x}[t \mapsto u_{t} \, \, \text{is continuous on} \, \, \mathbb{R}_{+}] = 1,\quad \text{for} \, \, \mathcal{E}-{\rm q.e.} \, x  \in K. \]
\end{cor}

The rest of this section will be devoted to show that for $\mathcal{E}$-q.e. $x \in K$, under $\mathbb{P}_{x}$, $(u_t)_{t\geq 0}$ solves \eqref{spde=1}, or rather its weaker form \eqref{formal1}.

In the sequel, we set $\Lambda_{1} := \Lambda + (\cdot,\cdot)_{L^{2}(\mu)}$ and $\mathcal{E}_{1} := \mathcal{E} + (\cdot,\cdot)_{L^{2}(\nu)}$, which are inner products for the Hilbert spaces $D(\Lambda)$ and $D(\mathcal{E})$ respectively. We shall also write in an abusive way, for \text{any} $\Phi \in C^{1}(H)$
\[ \mathcal{E}_{1}(\Phi, \Phi) :=  \int_{K} \Phi^{2} \d \nu + \frac{1}{2} \int_{K} \| \nabla \Phi \|_{H}^{2}. \]
Since the Dirichlet form $(\mathcal{E},D(\mathcal{E}))$ is quasi-regular, by the transfer method stated in VI.2 of \cite{ma2012introduction}, we can apply several results of \cite{fukushima2010dirichlet} in our setting. An important technical point is the density of the space $\mathcal{S}$ introduced in Section \ref{sect_ibpf_exp_func} above in the domain $D(\mathcal{E})$ of this Dirichlet form. To state this precisely, we consider $\mathscr{S}$ to be the vector space generated by functionals $F:H \to \mathbb{R}$ of the form
\[ F(\zeta) = \exp(- \langle \theta,\zeta^{2} \rangle), \quad \zeta \in H,  \] 
for some $\theta : [0,1] \to \mathbb{R}_{+}$ Borel and bounded. Note that $\mathscr{S}$ may be seen as a subspace of the space  $\mathcal{S}$ of Section \ref{sect_ibpf_exp_func} in the following sense: for any $F \in \mathscr{S}$, $F \rvert_{C([0,1])} \in \mathcal{S}$. We also set:
\[ \mathscr{S}_{K} := \{ F \big \rvert_{K}, \ F \in \mathscr{S} \}.  \]

\begin{lm}\label{density}
$\mathscr{S}_{K}$ is dense in $D(\mathcal{E})$.
\end{lm}
The proof of Lemma \ref{density} is postponed to Appendix \ref{Proofs}.

\subsection{Convergence of one-potentials}
The key tool in showing that the Markov process constructed above defines a solution of \eqref{formal1} is the IBPF \eqref{exp_fst_part_ibpf_a_b_1}. The rule of thumb is that the last term in the IbPF yields the expression of the drift in the SPDE. Recall however that, for any fixed $h \in C^{2}_{c}(0,1)$, the last term in \eqref{exp_fst_part_ibpf_a_b_1} is given by
\[\frac14\int_{0}^{1} \d r \, h_r\, \frac{{\rm d}^{2}}{{\rm d} a^{2}} \, \Sigma^1_r (\Phi(X) \, | \, a)\,  \biggr\rvert_{a=0}, \quad \Phi \in \mathcal{S}, \]
which defines a generalized functional in the sense of Schwartz, rather than a genuine measure, on $C([0,1])$. It is therefore not immediate to translate the IbPF in terms of the corresponding dynamics. The strategy we follow to handle this difficulty relies on Dirichlet form techniques: we approximate the above generalized functional by a sequence of measures admitting a smooth density w.r.t. the law of the reflecting Brownian bridge, and show that the corresponding one-potentials converge in the domain $D(\mathcal{E})$ of the Dirichlet form (see Section 5 of \cite{fukushima2010dirichlet} for the definition of one-potentials). This will imply that the associated additive functionals converge to the functional describing the drift in the SPDE.

More precisely, let $\rho$ be a smooth function supported on $[-1,1]$ such that
\[\rho \geq 0, \quad \int_{-1}^{1} \rho = 1, \quad \rho(y) = \rho(-y), \quad y \in \mathbb{R}.\] 
For all $\epsilon>0$, let
\begin{equation}
\label{def_mollifier}
\rho_{\epsilon}(y) := \frac{1}{\epsilon} \, \rho \left( \frac{y}{\epsilon} \right), \quad y \in \mathbb{R}. 
\end{equation}
Then, for all $\Phi \in \mathcal{S}$ and $h \in C^{2}_{c}(0,1)$, the right-hand side of the IbPF \eqref{exp_fst_part_ibpf_a_b_1} can be rewritten as follows
\begin{equation}
\label{relation_cond_lt}
\frac{1}{4} \int_{0}^{1} h_{r} \, \frac{ \d^{2}}{\d a^{2}} \, \Sigma^1_r (\Phi(X) \, | \, a)\, \biggr\rvert_{a=0} \d r = \frac{1}{2} \, \lim_{\epsilon \to 0} \mathbb{E} \left[ \Phi(|\beta|)  \int_{0}^{1} h_{r} \, \rho_{\epsilon}''(\beta_{r}) \, \d r \right]. 
\end{equation}
Indeed, starting from the right-hand side, by conditioning on the value of $|\beta_{r}|$, and recalling that $|\beta| \overset{(d)}{=} \nu$, the equality follows at once.

We will now show that the convergence of measures \eqref{relation_cond_lt} can be enhanced to a convergence in the space $D(\Lambda)$ of the associated one-potentials. We henceforth fix a function $h \in C^{2}_{c}(0,1)$. Then there exists $\delta \in (0,1)$ such that $h$ is supported in $[\delta, 1-\delta]$. For all $\epsilon > 0 $, let $G_{\epsilon}:H\to{\mathbb R}$ be defined by
\begin{equation}\label{Geps} 
G_{\epsilon}(z) := \frac{1}{2} \int_{0}^{1} h_{r} \, \rho_{\epsilon}''(z_{r}) \d r, \quad z \in H. 
\end{equation}
For all $t > 0$ and $z \in H$, we have
\[ \mathbf{Q}_{t} G_{\epsilon}(z) = \int_{0}^{1} \frac{h_{r}}{2\sqrt{2 \pi q_{t}(r)}} \int_{\mathbb{R}} \rho_{\epsilon}''(a) \exp \left( - \frac{(a-z(t,r))^{2}}{2 q_{t}(r)} \right) \d a \d r,
\] 
which, after two successive integration by parts, can be also written
\[ \int_{0}^{1} \frac{h_{r}}{2\sqrt{2 \pi q_{t}(r)}} \int_{\mathbb{R}} \rho_{\epsilon}(b) \left[ \left( \frac{b-z(t,r)}{q_t(r)} \right)^2 - \frac{1}{q_t(r)} \right]  \exp \left( - \frac{(b-z(t,r))^{2}}{2 q_{t}(r)} \right) \d b \d r, \] 
where $z(t,\cdot)$ depends on $z$ via \eqref{expr_solution_he}.
For all $\epsilon >0$, we define the functional $U_{\epsilon}: H \to \mathbb{R}$ by
\[ U_{\epsilon}(z) = \int_{0}^{\infty} e^{-t} \, \mathbf{Q}_{t}G_{\epsilon}(z) \d t, \quad z \in H.\]
Note that $U_{\epsilon}$ is the one-potential of the additive functional 
\[ \int_{0}^{t} G_{\epsilon}(v(s,\cdot)) \d s, \qquad t \geq 0, \]
associated with the Markov process $(v(t,\cdot))_{t\geq 0}$ in $H$ defined in \eqref{solution_she} (see Section 5 of \cite{fukushima2010dirichlet} for this terminology). In particular, $U_{\epsilon} \in D(\Lambda)$. For all $t>0$, let $G^{(t)}: H \to \mathbb{R}$ be the functional defined by
\[ G^{(t)}(z) := \int_{0}^{1} \frac{h_{r}}{2 \sqrt{2 \pi q_{t}(r)}} \left[ \left( \frac{z(t,r)}{q_t(r)} \right)^2 - \frac{1}{q_t(r)} \right] \exp \left( - \frac{z(t,r)^{2}}{2 q_{t}(r)} \right) \d r, \qquad z \in H. \] 
We claim that the following holds:

\begin{prop}
\label{conv_one_pot}
The functional $U :  H \to \mathbb{R}$ defined by 
\begin{equation}
\label{limiting_one_pt} 
U(z) := \int_0^\infty e^{-t} \, G^{(t)}(z) \d t, \qquad z \in H,
\end{equation}
belongs to $D(\Lambda)$. Moreover, $U_{\epsilon} \underset{\epsilon \to 0}{\longrightarrow} U$ in $D(\Lambda)$. 
\end{prop}

\begin{proof}
First note that $U_\epsilon \underset{\epsilon \to 0}{\longrightarrow} U$ in $L^2(\mu)$. Indeed, for all fixed $t>0$ and $z \in H$, we have
\[
\begin{split}
& |\mathbf{Q}_{t} G_\epsilon (z) - G^{(t)}(z)| \leq  
\\ & \int_{0}^{1} \frac{|h_{r}|}{2 \sqrt{2 \pi q_{t}(r)^{3}}} \int_{\mathbb{R}} \rho(x) \left| F\left( \frac{\epsilon x - z(t,r)}{\sqrt{q_t(r)}} \right) - F \left( \frac{z(t,r)}{\sqrt{q_t(r)}} \right) \right| \d x \d r, \end{split}
 \]
where the function $F: \mathbb{R} \to \mathbb{R}$ is defined by
\[F(y) = (y^2-1) \exp(-y^2/2), \qquad y \in \mathbb{R}. \]
Since $F$ is continuous and bounded, by dominated convergence, we deduce that, for all $r \in (0,1)$ and $x \in \mathbb{R}$
\[ \left \| F\left( \frac{\epsilon x - z(t,r)}{q_t(r)} \right) - F \left( \frac{z(t,r)}{q_t(r)} \right) \right \|_{L^2(\mu)} \underset{\epsilon \to 0}{\longrightarrow} 0.  \] 
Therefore, again by dominated convergence, we have
\[\begin{split}
&\|\mathbf{Q}_{t} G_\epsilon - G^{(t)}\|_{L^2(\mu)} \\
&\leq \int_{0}^{1} \frac{|h_{r}|}{2 \sqrt{2 \pi q_{t}(r)^{3/2}}} \int_{\mathbb{R}} \rho(x) \left\| F\left( \frac{\epsilon x - z(t,r)}{q_t(r)} \right) - F \left( \frac{z(t,r)}{q_t(r)} \right) \right\|_{L^2(\mu)} \d x \d r\\
& \underset{\epsilon \to 0}{\longrightarrow} 0.  
\end{split}\]
Recall that we have fixed $\delta \in (0,1)$ such that $h$ is supported in $[\delta,1-\delta]$. As showed in the proof of Proposition 1 in \cite{zambotti2004occupation}, there exists $C_\delta>0$ such that, for all $r \in (\delta, 1-\delta)$ and $t>0$ 
\begin{equation}
\label{lower_bound_qt}
q_t(r) \geq C_\delta (\sqrt{t} \wedge 1).
\end{equation}
In the following, we will denote by $C_\delta$ any constant depending only on $\delta$, and whose value may change from line to line. Thanks to \eqref{lower_bound_qt}, we obtain the bound
\[\|\mathbf{Q}_{t} G_\epsilon - G^{(t)}\|_{L^2(\mu)} \leq C_\delta \|F \|_{\infty} \frac{\|h\|_\infty}{t^{3/4} \wedge 1}, \]
where the right-hand side is integrable w.r.t. the measure $e^{-t} \d t$ on $\mathbb{R}_+$.
Hence, by dominated convergence,
\[\begin{split}
\| U_{\epsilon} - U \|_{L^2(\mu)} & \leq \int_0^\infty e^{-t} \, \|\mathbf{Q}_{t} G_\epsilon - G^{(t)}\|_{L^2(\mu)} \d t \\
&\underset{\epsilon \to 0}{\longrightarrow} 0,
\end{split} \]
whence the claim. Now, we show that $U \in D(\Lambda)$. Note that, for all $t>0$ and $\epsilon >0$, we have
\[ \nabla \mathbf{Q}_{t} G_\epsilon (z) = \frac{1}{2} \int_{0}^{1} h_{r} \, g_t(r,\cdot) \, \mathbb{E}[ \rho_{\epsilon}^{(3)}(v(t,r))] \d r, \qquad z \in H, \]
where $v$ is given by \eqref{expr_solution_she} and where we are taking expectation with respect to the white noise $\xi$. Therefore, denoting by $\| \cdot \|_{L^2}$ the norm in $L^2(H,\mu; H)$, we have
\[\begin{split}
&\| \nabla \mathbf{Q}_{t} G_\epsilon \|^2_{L^2} = \\
& \frac{1}{4} \int_{[0,1]^2} h_{r} h_s \langle g_t(r,\cdot), g_t(s, \cdot) \rangle \, \int_H   \mathbb{E}[ \rho_{\epsilon}^{(3)}(v(t,r))] \, \mathbb{E}[ \rho_{\epsilon}^{(3)}(v(t,s))] \, \d \mu(z) \, \d r \d s 
\end{split} \]
where the integral in $\d \mu(z)$ is taken with respect to $v(0,\cdot) = z$. Hence
\[ \begin{split}
&\| \nabla \mathbf{Q}_{t} G_\epsilon \|^2_{L^2} = \\
& \int_{[0,1]} \frac{h_r h_s \langle g_t(r,\cdot), g_t(s, \cdot) \rangle}{4} \, \int_{\mathbb{R}^2} \rho^{(3)}_{\epsilon}(x) \,  \rho^{(3)}_{\epsilon}(y) \, \Gamma_{r,s} (x,y) \, \d x \d y \, \d r \d s  = \\
& \int_{[0,1]} \frac{h_r h_s \langle g_t(r,\cdot), g_t(s, \cdot) \rangle}{4} \, \int_{\mathbb{R}^2} \rho_{\epsilon}(x) \, \rho_{\epsilon}(y) \, \frac{\partial^6 \Gamma_{r,s}}{\partial x^3 \partial y^3} (x,y) \, \d x \d y \, \d r \d s, 
\end{split}\] 
where, for all $(r,s) \in [0,1]$ and $(x,y) \in \mathbb{R}^{2}$
\[
\Gamma_{r,s}(x,y) := \mathbb{E} \left[ \frac{1}{2 \pi \sqrt{q_{t}(r) q_{t}(s)}} \exp \left(- \frac{(x-z(t,r))^{2}}{2q_{t}(r)} - \frac{(y-z(t,s))^{2}}{2q_{t}(s)} \right) \right], 
\]
where $z(t,\cdot)$ is given by \eqref{expr_solution_he}, and where we are taking expectation with respect to $z \overset{(d)}{=} \mu$. Reasoning as in Section 6 of \cite{zambotti2005integration}, we see that $\Gamma_{r,s}$ is the density of the centered Gaussian law on $\mathbb{R}^{2}$ with covariance matrix
\[ M = \begin{pmatrix}
q_{\infty}(r) & q^{t}(r,s) \\
q^{t}(r,s) & q_{\infty}(s)
\end{pmatrix}. \]
Similarly, we have 
\[\begin{split}
\| \nabla G^{(t)}\|^2_{L^2} = 
\int_{0}^{1} \int_{0}^{1} \frac{h_{r} \, h_s \, \langle g_t(r,\cdot), g_t(s, \cdot) \rangle}{4} \,  \frac{\partial^6 \Gamma_{r,s}}{\partial x^3 \, \partial y^3} (0,0) \, \d r \d s. 
\end{split}\] 
So there remains to obtain a bound on 
\[  \underset{\mathbb{R}^{2}}{\sup} \ \left| \frac{\partial^{6}\Gamma_{r,s}}{\partial x^{3} \, \partial y^{3}} \right|, \]
for all $(r,s) \in [0,1]^{2}$. To do so, 
we use the following lemma:

\begin{lm}
\label{bound_derivative_density}
Let $f: \mathbb{R}^{2} \to \mathbb{R}$ be the density of a centered Gaussian law on $\mathbb{R}^{2}$ with non-degenerate covariance matrix $M$ satisfying $|M_{i,j}| \leq 1$ for all $i,j \in \{1,2\}$. Then, for all $k,\ell \in \mathbb{N}$ and $(x,y) \in \mathbb{R}^{2}$
\[ \left| \frac{\partial^{k+\ell} f}{\partial^{k} x \, \partial^{\ell} y} \right| \leq A_{k,\ell} \  \det(M)^{-\frac{1+k+\ell}{2}}\]
where $A_{k, \ell} >0$ is a constant depending only on $k$ and $\ell$.
\end{lm}

\begin{proof}
Setting
\[ M = \begin{pmatrix}
a & b \\
b & c
\end{pmatrix},  \]
we can express the eigenvalues $\lambda$ and $\mu$ of $M$ as
\[ \lambda = \frac{a+c}{2} + \sqrt{\left(\frac{a-c}{2} \right)^2 + b^2} \]
and
\[ \mu = \frac{a+c}{2} - \sqrt{\left(\frac{a-c}{2} \right)^2 + b^2}. \]
Hence, since $a$, $b$ and $c$ are bounded by $1$, we deduce that $\lambda$ and $\mu$ are bounded by some universal constant $C>0$. 
Let now $P$ be an orthogonal matrix such that $M = P^{T} D P$, where $P^{T}$ denotes the transposed of the matrix $P$, and where
\[ D = \begin{pmatrix}
\lambda & 0 \\
0 & \mu
\end{pmatrix}.\]  
Then, for all $u \in \mathbb{R}^2$
\begin{equation}
\label{expre_density}
f(u) = \frac{1}{2 \pi \sqrt{\det(M)}} \, g(Pu) 
\end{equation}
where 
\[g(v) := \exp \left(-\frac{1}{2} v^T D^{-1} v \right) = \exp \left(-\frac{x^2}{2 \lambda} - \frac{y^2}{2 \mu} \right)\]
for all $v = (x,y) \in \mathbb{R}^2$. Since the function $u \mapsto e^{-\frac{u^2}{2}}$ is bounded on $\mathbb{R}$ with all its derivatives, we deduce that for all $k, \ell \in \mathbb{N}$, there exists $C_{k, \ell}>0$ depending only on $k$ and $\ell$ such that
\[\left| \frac{\partial^{k+\ell} g}{\partial x^{k} \partial y^{\ell}} \right| \leq C_{k,\ell} \, \lambda^{-k/2} \mu^{-\ell/2}. \]
Therefore, since $\lambda$ and $\mu$ are bounded by $C$, and noting that $\det(M) = \lambda \, \mu$, setting $C'_{k,\ell} := C_{k,\ell} \, C^{\frac{k+\ell}{2}}$ we have
\[\left| \frac{\partial^{k+\ell} g}{\partial x^{k} \partial y^{\ell}} \right| \leq C'_{k,\ell} \, \det(M)^{-\frac{k+\ell}{2}}. \]
Hence, by the relation \eqref{expre_density} and the chain rule, and since the coefficients of the orthogonal matrix $P$ are all bounded aby $1$, we obtain the claim. 
\end{proof}
We now apply the Lemma to the Gaussian density function $\Gamma_{r,s}$ for all $(r,s) \in (0,1)$. Note that $q_\infty(r) \leq 1$ and $q_{\infty}(s) \leq 1$, so all coefficients of its covariance matrix $M$ are indeed bounded by $1$ as requested. Therefore 
\[  \underset{\mathbb{R}^{2}}{\sup} \ \left| \frac{\partial^{6}\Gamma_{r,s}}{\partial x^{3} \partial y^{3}} \right| \leq A \, \det(M)^{-7/2}, \]
where $A \in (0, \infty)$ is a universal constant. Now
\[
\det(M) = q_{\infty}(r) q_{\infty}(s) - q^{t}(r,s)^{2},
\]
But
\[\begin{split}
q_{\infty}(r) q_{\infty}(s) - q^{t}(r,s)^{2} &\geq q_{\infty}(r) q_{\infty}(s) - q_{\infty}(r,s)^{2} \\
&= r(1-r)s(1-s) - (r \wedge s - rs)^{2} \\
&= s \wedge r (1- s \vee r) |s-r|,
\end{split}\] 
so we obtain the lower bound
\begin{equation}
\label{lower_bound_det_trivial}
\det(M) \geq \delta^{2} |r-s| 
\end{equation}
for all $r,s \in [\delta,1-\delta]$. On the other hand, reasoning as in Section 6 of \cite{zambotti2005integration}, we can show that there exists $c_{\delta}>0$ depending only on $\delta$ such that, for all $r,s \in [\delta,1-\delta]$ 
\[q_{\infty}(r) q_{\infty}(s) - q^{t}(r,s)^{2} \geq c_{\delta} \, (t \wedge 1)^{1/2},\]
which yields the lower bound
\begin{equation}
\label{lower_bound_det} 
\det(M) \geq c_{\delta} \, (t \wedge 1) ^{1/2}.
\end{equation} 
As a consequence, for all $r,s \in [\delta,1-\delta]$, interpolating \eqref{lower_bound_det_trivial} and \eqref{lower_bound_det}, we thus obtain 
\begin{equation}
\label{bound_derivative_gamma}
\left| \frac{\partial^{6} \Gamma_{r,s}}{\partial^{3} x \, \partial^{3}y} \right| \leq C_\delta \, (t \wedge 1)^{-\gamma/2} |r-s|^{-(7/2-\gamma)},
\end{equation} 
for any $\gamma \in (5/2,3)$, where $C_\delta>0$ is a constant depending only on $\delta$. Note also that, for some universal constant $C>0$, we have
\begin{equation} 
\label{bound_green}
\forall r,s \in [\delta, 1-\delta], \qquad \langle g_t(r,\cdot), g_t(s, \cdot) \rangle = g_{2t}(r,s) \leq C \, t^{-1/2}, 
\end{equation}
see e.g. Exercise 4.16 in \cite{zambotti2017random}. Therefore,
\[\| \nabla G^{(t)}\|^2_{L^2} \leq 
 C_\delta \, \|h\|_{\infty}^2(t \wedge 1)^{-(1+\gamma)/2} \, \int_{0}^{1} \int_{0}^{1} |r-s|^{-(7/2-\gamma)} \, \d r \d s, \] 
and the last integral is finite due to the choice of $\gamma$. Therefore, we deduce that $\| \nabla G^{(t)}\|_{L^2} \leq C(\delta,h,\gamma) \, t^{-(1+\gamma)/4}$, where the constant $C(\delta,h,\gamma)$ does not depend on $t$. Since $(1+\gamma)/4<1$, it follows that
\[ \int_0^\infty e^{-t} \, \| \nabla G^{(t)}\|_{L^2} \d t < \infty, \]
so that $\nabla U \in L^2(H,\mu; H)$. Therefore $U \in D(\Lambda)$ as claimed. There remains to prove that $U_\epsilon \underset{\epsilon \to 0}{\longrightarrow} U$ in $D(\Lambda)$. Note that, for all $t>0$ and $\epsilon > 0$,
\[\begin{split}
&\|\mathbf{Q}_{t} G_\epsilon - G^{(t)}\|^2_{L^2} \\
&= \int_{[0,1]^2} \d r \d s \, \frac{h_{r} h_s \langle g_t(r,\cdot), g_t(s, \cdot) \rangle}{4} \int_{\mathbb{R}^2} \d x \d y \, \rho(x) \, \rho(y) \, \Gamma^{(3;3)}_{r,s}(\epsilon x , \epsilon y ),
\end{split}\] 
where for all $(u,v) \in \mathbb{R}^2$,
\[\Gamma^{(3;3)}_{r,s}(u,v) := \frac{\partial^{6} \Gamma_{r,s}}{\partial^{3} x \, \partial^{3}y} (u,v)  - \frac{\partial^{6} \Gamma_{r,s}}{\partial^{3} x \, \partial^{3}y} (u,0) - \frac{\partial^{6} \Gamma_{r,s}}{\partial^{3} x \, \partial^{3}y} (0,v) + \frac{\partial^{6} \Gamma_{r,s}}{\partial^{3} x \, \partial^{3}y} (0,0).\]
By \eqref{bound_derivative_gamma} and \eqref{bound_green} 
we deduce that
\[ \|\mathbf{Q}_{t} G_\epsilon - G^{(t)}\|^2_{L^2} \leq  C_\delta \,  \|h\|_{\infty}^2 \, t^{-(1+\gamma)/2} \, \int_{0}^{1} \int_{0}^{1} |r-s|^{-(7/2-\gamma)} \, \d r \d s,\]
so that:
\[ \|\mathbf{Q}_{t} G_\epsilon - G^{(t)}\|_{L^2} \leq C(\delta, h, \gamma) \, t^{-(1+\gamma)/4},\]
where $C(\delta,h, \gamma)>0$ is independent of $\epsilon $ and $t$. Recall that the right-hand side above is integrable with respect to $ e^{-t} \d t$. Moreover, since $\Gamma_{r,s}$ is continuous, it follows that for all $t>0$, 
\[ \|\mathbf{Q}_{t} G_\epsilon - G^{(t)}\|_{L^2} \underset{\epsilon \to 0}{\longrightarrow} 0. \]
Hence, by dominated convergence, we deduce that 
\[ \| \nabla U_\epsilon - \nabla U \|_{L^2} \leq \int_0^t e^{-t} \, \| \nabla P_t G_\epsilon - \nabla G^{(t)}\|_{L^2} \, \d t  \underset{\epsilon \to 0}{\longrightarrow} 0. \]
Hence  $U_{\epsilon} \underset{\epsilon \to 0}{\longrightarrow} U$ in $D(\Lambda)$, and the Proposition is proved.
\end{proof}

\subsection{A projection principle}

Note that in the above section we worked in the domain $D(\Lambda)$ of the Dirichlet form associated with the Brownian bridge. For our dynamical problem, we shall however need to transfer the above results to the domain $D(\mathcal{E})$ of the Dirichlet form associated with the Bessel bridge. To do so, we invoke the following projection principle, which was first used in \cite{zambotti2004occupation} for the case of a $3$-Bessel bridge (see Lemma 2.2 therein). Recall the notations $\Lambda_{1} := \Lambda + (\cdot,\cdot)_{L^{2}(\mu)}$ and $\mathcal{E}_{1} := \mathcal{E} + (\cdot,\cdot)_{L^{2}(\nu)}$.

\begin{lm}
\label{projection}
There exists a unique bounded linear operator $\Pi:  D(\Lambda) \to D(\mathcal{E})$ such that, for all $F,G \in D(\Lambda)$ and $f \in D(\mathcal{E})$
\[ 
\Lambda_{1}(F,f \circ j) = \mathcal{E}_{1}(\Pi F,f), 
\]
where $j$ is as in \eqref{absolute_value_map}. Moreover, we have
\[\mathcal{E}_{1}(\Pi F, \Pi F) \leq \Lambda_{1}(F,F). \]
\end{lm}

\begin{proof}
We use the same arguments as in the proof of Lemma 2 in \cite{zambotti2004occupation}. Let $\mathcal{D} := \{ \varphi \circ j, \quad \varphi \in D(\mathcal{E}) \}$. By Proposition \ref{closability}, $\mathcal{D}$ is a linear subspace of $D(\Lambda)$ which is isometric to $D(\mathcal{E})$. 
In particular, it is a closed subspace of the Hilbert space $D(\Lambda)$. Hence, we may consider the orthogonal projection operator $\hat{\Pi}$ onto $\mathcal{D}$. Then, for all $F \in D(\Lambda)$, let $\Pi F $ be the unique element of $D(\mathcal{E})$ such that $\hat{\Pi} F = (\Pi F) \circ j$. It then follows that $\Pi$ possesses the required properties.
\end{proof}

We obtain the following refinement of the IbPF \eqref{exp_fst_part_ibpf_a_b_1} for $P^1$.
\begin{cor}
Let $U$ be as in \eqref{limiting_one_pt}. For all $f \in D(\mathcal{E})$ and $h \in C^{2}_{c}(0,1)$, we have
\begin{equation}
\label{IbPF_Dirichlet}
\mathcal{E}\left(\langle h, \cdot \rangle - \frac{1}{2} \Pi U \, ,\, f\right) = - \frac{1}{2} \int_{K} \left(\langle h'', \zeta \rangle - \Pi  U(\zeta)\right) f(\zeta) \d \nu(\zeta).
\end{equation}
\end{cor}
\begin{proof}
By the density of $\mathscr{S}_{K}$ in $D(\mathcal{E})$ proved in Lemma \ref{density}, it is enough to consider
$f \in \mathscr{S}_{K}$. By \eqref{relation_cond_lt}
\[ \begin{split}
&\frac{1}{4} \int_{0}^{1} {\rm d} r\, h_{r} \frac{\d^{2}}{\d a^{2}} \Sigma^1_r\left(f(X) \,|\, a\right) \, \biggr \rvert_{a=0} = \frac{1}{2} \, \lim_{\epsilon \to 0} \mathbb{E} \left[ f(|\beta|)  \int_{0}^{1} h_{r} \, \rho_{\epsilon}''(\beta_{r}) \, \d r \right] \\
&= \lim_{\epsilon \to 0} \int (f\circ j) \, G_{\epsilon} \d \mu 
= \lim_{\epsilon \to 0} \, \Lambda_{1} ( f \circ j , \, U_{\epsilon}) 
=  \, \Lambda_{1} ( f \circ j, \, U) = \, \mathcal{E}_{1} (f , \, \Pi U).
\end{split}
\]
Therefore, for all $f \in \mathscr{S}_{K}$, the IbPF \eqref{exp_fst_part_ibpf_a_b_1} can be rewritten
\[ 2 \mathcal{E} (\langle h, \cdot \rangle, f ) = -  \int_{K} \langle h'', \zeta \rangle \, f(\zeta) \d \nu(\zeta) + \, \mathcal{E}_{1} (f , \, \Pi U),  \]
that is
\[  \mathcal{E} \left(\langle h, \cdot \rangle - \frac{1}{2} \Pi U, f \right) = - \frac{1}{2} \int_{K} (\langle h'', \zeta \rangle - \Pi  U(\zeta)) \, f(\zeta) \d \nu(\zeta).  \]  
The proof is complete.
\end{proof}

Recall that $M=(\Omega, \mathcal{F}, (u_t)_{t \geq 0}, (\mathbb{P}_x)_{x \in K})$ denotes the Markov process properly associated with the Dirichlet form $(\mathcal{E},D (\mathcal{E}))$ constructed above. 
Note that, by Theorem 5.2.2 in \cite{fukushima2010dirichlet}, for all $F \in D(\mathcal{E})$, we can write in a unique way
\[ 
F(u_{t}) - F(u_{0}) = M^{[F]}_{t} + N^{[F]}_{t}, \quad t \geq 0, 
\]
$\mathbb{P}_{\nu}$ a.s., where $M^{[F]}$ is a martingale additive functional, and $N^{[F]}$ is an additive functional of zero energy. Using this fact we can thus write $u$ as the weak solution to some SPDE, but with coefficients that are not explicit. However the formula \eqref{IbPF_Dirichlet} above will allow us to identify these coefficients. 

We can now finally state the result justifying that that the Markov process constructed above satisfies the SPDE \eqref{formal1} above.

\begin{thm}\label{fukushima_decomposition}
For all $h \in C^{2}_{c}(0,1)$, we have 
\[ \langle u_{t}, h \rangle - \langle u_0, h \rangle = M_{t} + N_{t}, \qquad \mathbb{P}_{u_0}-\text{a.s.}, \quad \text{q.e.} \ u_0 \in K.  \]
Here $(N_{t})_{t \geq 0}$ is a continuous additive functional of zero energy satisfying
\[N_{t} - \frac{1}{2} \int_{0}^{t} \langle h'', u_{s} \rangle \d s = \underset{\epsilon \to 0}{\lim} \, N^{\epsilon}_{t}, \qquad
N^{\epsilon}_{t} := -\frac{1}{4} \int_{0}^{t} \langle \rho''_{\epsilon}(u_{s}), h \rangle \d s, 
\]
in $\mathbb{P}_{\nu}$-probability, uniformly in $t$ on finite intervals.
Moreover, $(M_{t})_{t \geq 0}$ is a martingale additive functional whose sharp bracket has the Revuz measure $\|h\|_{H}^{2} \, \nu$. Finally we also have
\[N_{t} - \frac{1}{2} \int_{0}^{t} \langle h'', u_{s} \rangle \d s = \underset{k \to \infty}{\lim} \, N^{\epsilon_k}_{t}
\]
along a subsequence $\epsilon_k\to 0$ in $\mathbb{P}_{u_0}$-probability, for q.e. $u_0 \in K$.  
\end{thm}

\begin{proof}
On the one hand, by \eqref{IbPF_Dirichlet}, we can write
\begin{equation} 
\label{fukushima_decomposition_one}
\langle u_{t}, h \rangle - \frac{1}{2} \Pi U (u_{t}) - \left(\langle u_0, h \rangle - \frac{1}{2} \Pi U (u_0) \right) = N^{(1)}_{t} + M^{(1)}_{t}, 
\end{equation}
where $N^{(1)}$ is the continuous additive functional of zero energy given by
\[
N^{(1)}_{t} = \frac{1}{2} \int_{0}^{t} \left(\langle h'', u_{s} \rangle - \Pi  U(u_{s}) \right) \d s, \quad t \geq 0
\]
and $M^{(1)}$ is defined by \eqref{fukushima_decomposition_one}. On the other hand, for all $\epsilon >0$, by definition of $U_{\epsilon}$, we have for $G_\epsilon$ as in \eqref{Geps}
\[ \Lambda_{1} ( U_{\epsilon}, \Phi) = \int_{H} G_{\epsilon} \,\Phi \d\mu, \quad \Phi \in D( \Lambda ). \]
Hence, remarking that $G_{\epsilon} = g_{\epsilon} \circ j$, where $g_{\epsilon}: K \to \mathbb{R}$ is the functional defined by
\[ g_{\epsilon}(z) := \frac12\int_{0}^{1} h_{r} \, \rho_{\epsilon}''(z_{r}) \d r = \frac12\langle \rho''_{\epsilon}(z), h \rangle, \]
by Lemma \ref{projection}, we obtain for all $f \in D( \mathcal{E} )$
\begin{equation}
\label{u_epsilon}
\mathcal{E}_{1} ( \Pi U_{\epsilon}, f) = \int_{K} f(z) \, g_{\epsilon}(z) \d \nu(z) = - \int_{K} f(z)(\Pi U_{\epsilon}(z) - g_{\epsilon}(z)) \d \nu(z).
\end{equation}
As a consequence, we have the decomposition
\begin{equation}
\label{fukushima_decomposition_two}
\frac{1}{2} \Pi U_{\epsilon} (u_{t}) - \frac{1}{2} \Pi U_{\epsilon} (u_0) = N^{(2,\epsilon)}_{t} + M^{(2,\epsilon)}_{t}, 
\end{equation}
where $N^{(2, \epsilon)}$ is the continuous additive functional of zero energy given by
\[N^{(2, \epsilon)}_{t} = \frac{1}{2} \int_{0}^{t}  \left( \Pi U_{\epsilon} (u_{s})- g_{\epsilon}(u_{s})\right) \d s, \quad t \geq 0\]
and $M^{(2,\epsilon)}$ is defined by \eqref{fukushima_decomposition_two}. 
Since $U_{\epsilon} \underset{\epsilon \to 0}{\longrightarrow} U$ in $D(\Lambda)$ by Proposition \ref{conv_one_pot}, by the continuity of $\Pi:D(\Lambda) \to D(\mathcal{E})$, we have the convergence $\Pi U_{\epsilon} \underset{\epsilon \to 0}{\longrightarrow} \Pi U$ in $D(\mathcal{E})$. Therefore, setting
\[ M^{(2)}_{t} = M^{[\Pi U]}_{t}, \qquad N^{(2)}_{t} := N^{[\Pi U]}_{t}, \]
then, by (5.1.1), (5.2.22) and (5.2.25) in \cite{fukushima2010dirichlet},
we have
\[ \Pi U_{\epsilon}(u_{t}) - \Pi U_{\epsilon}(u_{0}) \underset{\epsilon \to0}{\longrightarrow} \Pi U(u_{t}) - \Pi U(u_{0}), 
\quad  M^{(2, \epsilon)}_{t} \underset{\epsilon \to0}{\longrightarrow} M^{(2)}_{t}, \quad  N^{(2, \epsilon)}_{t} \underset{\epsilon \to \infty}{\longrightarrow} N^{(2)}_{t} \]
in $\mathbb{P}_{\nu}$-probability, for the topology of uniform convergence on finite intervals of $t \in \mathbb{R}_{+}$. 
Adding equality \eqref{fukushima_decomposition_two} to \eqref{fukushima_decomposition_one} yields 
\[ \langle u_{t}, h \rangle - \langle u_0, h \rangle = M_{t} + N_{t}, \]
with $M_{t} = M^{1}_{t} + M^{2}_{t}$
and
\[\begin{split}
N_{t} &= N^{1}_{t} + N^{2}_{t} = \frac{1}{2} \int_{0}^{t} \left(\langle h'', u_{s} \rangle - \Pi  U(u_{s}) \right) \d s + \underset{\epsilon \to 0}{\lim} \, \frac{1}{2} \int_{0}^{t}  \left( \Pi U_{\epsilon} (u_{s})- g_{\epsilon}(u_{s})\right) \d s \\
&= \frac{1}{2} \int_{0}^{t} \langle h'', u_{s} \rangle \d s -  \underset{\epsilon \to 0}{\lim} \, \frac{1}{2} \int_{0}^{t}  g_{\epsilon}(u_{s}) \d s,
\end{split}\]
Moreover, note that $M=M^{[F_h]}$, where $F_h \in D(\mathcal{E})$ is given by
\[ F_h(z) := \langle z, h \rangle, \quad z \in K. \]
Hence, by Theorem 5.2.3 in \cite{fukushima2010dirichlet}, $\mu_{<M>}$ is given by $\|h\|_{L^{2}(0,1)}^{2} \cdot \nu$. For the last statement, we apply \cite[Corollary 5.2.1]{fukushima2010dirichlet}. 
\end{proof}

\subsection{A distinction result}

As a consequence of our IbPFs and the above constructions, we can prove that the Markov process $(u_t)_{t\geq 0}$ constructed above is not identically equal in law to the process corresponding to the modulus of the solution $(v_t)_{t\geq 0}$ to the stochastic heat equation, as one could be tempted to infer in analogy with the relation between the invariant measures $\mu$ and $\nu$.

Let $K^{\mathbb{R}_+}$ denote the space of functions from $\mathbb{R}_{+}$ to $K$, endowed with the product $\sigma$-algebra. For all $x \in K$, let $P_{x}$ be the law, on $K^{\mathbb{R}_+}$, of the Markov process $(u_t)_{t \geq 0}$ associated with $\mathcal{E}$, started from $x$.
Similarly, for all $z \in H$, let $\mathbf{P}_{z}$ be the law, on $K^{\mathbb{R}_+}$, of $(|v_{t}|)_{t \geq 0}$, where $(v_{t})_{t \geq 0}$ is the solution of the stochastic heat equation \eqref{solution_she}, with $v_{0}=z$. 

\begin{thm}
\label{dist_res}
\[\mu \left( \{ z \in H : \, P_{|z|} \neq \mathbf{P}_{z} \} \right) > 0. \] 
\end{thm}

\begin{proof}
Assume by contradiction that $P_{|z|} = \mathbf{P}_{z}$ for $\mu$-a.e. $z \in H$. Then, recalling that $(\mathbf{Q}_{t})_{t \geq 0}$ denotes the semigroup associated with $\Lambda$, and $(Q_{t})_{t \geq 0}$ the semigroup associated with $\mathcal{E}$, we would have
\[ \mathbf{Q}_{t}(f \circ j) = (Q_{t} f) \circ j, \quad \mu - \text{a.e.}, \]  
for all $t \geq 0$ and $f \in L^{2}(\nu)$. Therefore, the corresponding families of resolvents $(\mathbf{R}_{\lambda})_{\lambda  >0}$ and $(R_{\lambda})_{\lambda >0}$ would satisfy, for all $f \in L^{2}(\nu)$ 
\[\mathbf{R}_{1} (f \circ j) = (R_{1} f) \circ j, \]
where the equality holds in $L^{2}(\mu)$. In particular, this shows that $(R_{1} f) \circ j \in D(\Lambda)$ for any $f$ as above. We then claim that, for all $F \in D(\Lambda)$, $\Pi F = \mathbb{E} [F(\beta) \, | \, |\beta| \,]$ $\mu$-a.e. Indeed, by the previous observations, for all $f \in L^{2}(\nu)$, it holds
\begin{equation}\label{pi_is_cond_exp}
\begin{split}
\int_{H} (f \circ j )(z) F(z) \d \mu(z) &= \Lambda_{1}( \mathbf{R}_{1}(f \circ j), F) = \Lambda_{1}( (R_{1}f) \circ j, F) 
\\&= \mathcal{E}_{1} (R_{1}f , \Pi F) = \int_{K} f(x) (\Pi F) (x) \d \nu(x),
\end{split}\end{equation}
i.e. $\Pi F = \mathbb{E} [F(\beta) \, | \, |\beta| \,]$ $\mu$-a.e., as claimed. By \eqref{pi_is_cond_exp} and the first equality in Lemma \ref{projection}, we deduce that, for all $f \in D(\mathcal{E})$ and $F \in D(\Lambda)$
\[\Lambda(F,f \circ j) = \mathcal{E}(\Pi F,f). \] 
Consider now the process $(v_{t})_{t \geq 0}$ associated with $\Lambda$ and started from $v_{0}=\beta$, where $\beta$ is a Brownian bridge on $[0,1]$. Consider also the process $(u_{t})_{t \geq 0}$ associated with $\mathcal{E}$ under the law $\mathbb{P}_{\nu}$ (so that, in particular, $u_{0}\overset{(d)}{=}|\beta|$). Thus the processes $v$ and $u$ are stationary, and $|v| \overset{(d)}{=} u$ by our assumption. Let us set
\[
A_t:=\langle |v_t| , h \rangle - \langle |v_0| , h \rangle - \frac{1}{2} \int_0^t\langle |v_{s}| , h'' \rangle \d s,
\]
\[
C_t:=\langle u_t , h \rangle - \langle u_0 , h \rangle - \frac{1}{2} \int_0^t \langle u_s , h'' \rangle \d s.
\]
Let further $k \in C^{2}([0,1])$ with $k(0)=k(1)=0$, and consider the functionals $\Psi_k:H\to\R$ and $\tilde\Psi_k:K\to\R$ given by
\[
\Psi_k(z):=\exp(\langle k,z\rangle), \qquad
\tilde\Psi_k(y):=\E\left[\Psi_k(\beta)\,|\, |\beta|=y\, \right], \qquad y\in K.
\]
Note that $\Psi_{k} \in D(\Lambda)$, and recall that, by the above remarks, $\tilde{\Psi}_{k} = \Pi \Psi_{k}$ $\mu$-a.e., so in particular $\tilde{\Psi}_{k} \in D(\mathcal{E})$. We then have
\[
\begin{split}
&J(t):=- \frac{d}{dt}\E\left[ A_t\, \tilde\Psi_k(|v_0|)\right]   = 
\\ & = - \frac{d}{dt}\E\left[ (\langle u_t , h \rangle - \langle u_0 , h \rangle )\, \tilde\Psi_k(u_0)\right] + \frac{1}{2} \frac{d}{dt} \mathbb{E} \left[\int_{0}^{t} \langle h'',|v_{s}|\rangle \d s \, \Psi_k(\beta)\right] \\
&= \mathcal{E}(\langle \cdot , h \rangle \, , \, \tilde \Psi_k) + \frac{1}{2} \E[\langle h'',|\beta|\rangle\Psi_k(\beta)]  
= \Lambda(\langle |\cdot| , h \rangle \, , \, \Psi_k) + \frac{1}{2} \E[\langle h'',|\beta|\rangle\Psi_k(\beta)] \\
& = \frac{1}{2} \E[\langle \nabla \Psi_k(\beta),{\rm sign}(\beta)\,h\rangle + \langle h'',|\beta|\rangle\Psi_k(\beta)] 
= \E\left[\Psi_k(\beta)\int_0^1 h \, :\dot{\beta}^2: \d L^0\right]
\end{split}
\]
by (3.10) in \cite{zambotti2005integration}, or rather its analogue for the Brownian bridge as stated in Remark 1.3 of \cite{grothaus2016integration}. But, by \cite[Corollary 3.4]{zambotti2005integration} and \cite[Theorem 3.2]{grothaus2016integration}, the last quantity equals
\[ \begin{split}
\sqrt{\frac{1}{2 \pi}}  \, e^{\frac{1}{2} \langle Qk, k \rangle} \int_{0}^{1} \frac{h_r}{\sqrt{r(1-r)}} \exp\left( - \frac{K_{r}^{2}}{2r(1-r)} \right) \lambda(K'_r, -K_{r}, r) \d r,
\end{split}
\]
where $K= Q k$, with $Q$ the covariance operator of $\beta$,
\[ 
(Q k)_{r} = \int_{0}^{1} (r \wedge \sigma - r \sigma) \, k_{\sigma} \d \sigma, \qquad r \in [0,1],
 \]
and $\lambda : \mathbb{R}^{2} \times [0,1] \to \mathbb{R}$ is defined by   
\[ \lambda(x,y,r) := x^{2} + xy \frac{1-2r}{r(1-r)}  +  y^{2}\frac{(1-2r)^{2}}{4r^{2}(1-r)^{2}} - \frac{1}{4r(1-r)}, \quad x,y \in \mathbb{R}, \ r \in [0,1]. 
\]
Hence, 
\begin{equation}
\label{quantity_one}
\begin{split}
J(t) =  \sqrt{\frac{2}{\pi}} e^{\frac{1}{2} \langle Qk, k \rangle} \int_{0}^{1} \frac{h_r}{\sqrt{r(1-r)}} \exp\left( - \frac{K_{r}^{2}}{2r(1-r)} \right) \lambda(K'_{r}, -K_{r}, r) \d r.
\end{split}
\end{equation}
On the other hand
\[
\begin{split}
L(t):=- \left. \frac{d}{dt}\E\left[ C_t\, \tilde\Psi_k(|v_0|)\right] \,\right|_{t=0} 
& = 
\mathcal{E} (\Pi \Psi_k , \langle \cdot, h \rangle) + \frac{1}{2} \E[\langle h'',|\beta|\rangle \, \Pi \Psi_k(|\beta|)] 
\\ = \frac{1}{2}\mathcal{E}(\Pi U,\Pi \Psi_{k})
&= \frac{1}{4} \, \underset{\epsilon \to 0}{\lim} \, \mathbb{E} \left[\int_{0}^{1} h_r \, \rho''_{\epsilon}(|\beta_r|) \d r \, \Pi \Psi_k(|\beta|) \right],
\end{split}
\]
where we used \eqref{IbPF_Dirichlet} to obtain the second equality, and the fact that $U = \underset{\epsilon \to 0}{\lim}  \, U_{\epsilon}$ in $D(\mathcal{E})$, combined with \eqref{u_epsilon}, to obtain the third one. Therefore, recalling that $\Pi \Psi_{k} = \E(\Psi_{k} \, | \, |\beta|)$ $\mu$-a.e., we have
\[\begin{split}
L(t) =  \frac14 \lim_{\epsilon\to 0} \mathbb{E} \left[\int_{0}^{1} h_r \, \rho''_{\epsilon}(|\beta_r|) \d r \, \Psi_{k}( \beta ) \right]
 = \frac14 \lim_{\epsilon\to 0} \mathbb{E} \left[ \int_{0}^{1} h_r \, \rho''_{\epsilon}(\beta_r) \d r \, e^{\langle k, \beta \rangle} \right].
\end{split}
\]
By the Cameron-Martin formula, for all $\epsilon >0$
\[ \begin{split}
&  \frac{1}{4} \, \mathbb{E} \left[\int_{0}^{1} h_r \, \rho''_{\epsilon}(\beta_r) \d r \, e^{\langle k, \beta \rangle} \right] = 
\\
&= \frac{1}{4} e^{\frac{1}{2} \langle Qk, k \rangle} \int_{0}^{1} \frac{h_r}{\sqrt{2 \pi r(1-r)}} \int_{\mathbb{R}} \rho_{\epsilon}''(a) \exp\left( - \frac{(a-K_{r})^{2}}{2r(1-r)} \right) {\rm d} a\d r \\
&\underset{\epsilon \to 0}{\to} \frac{1}{4} e^{\frac{1}{2} \langle Qk, k \rangle} \int_{0}^{1} \frac{h_r}{\sqrt{2 \pi r(1-r)}} \left[\frac{K_{r}^{2} - r(1-r)}{r^{2}(1-r)^{2}}\right] \exp\left( - \frac{K_{r}^{2}}{2r(1-r)} \right){\rm d}r.
\end{split}\]
Hence we obtain
\begin{equation} 
\label{quantity_two}
\begin{split}
&L(t) = \frac{1}{4} e^{\frac{1}{2} \langle Qk, k \rangle} \int_{0}^{1} \frac{h_r}{\sqrt{2 \pi r(1-r)}} \left[\frac{K_{r}^{2} - r(1-r)}{r^{2}(1-r)^{2}}\right] \exp\left( - \frac{K_{r}^{2}}{2r(1-r)} \right){\rm d}r.
\end{split}
\end{equation}
Since $|v|$ and $u$ have the same law, $J(t)=L(t)$ and therefore the right-hand sides of \eqref{quantity_one} and \eqref{quantity_two} above are equal. This being true for any $h \in C^{2}_{c}(0,1)$, we deduce that
\[ \frac{K_{r}^{2} - r(1-r)}{4 r^{2}(1-r)^{2}} = \lambda(K'_{r}, -K_{r}, r),  \] 
for a.e. $r \in (0,1)$, hence for all $r$ by continuity. We thus deduce that 
\begin{equation*} 
(K'_{r})^{2} - \frac{1-2 r}{r(1-r)} K_{r}K'_{r} - \frac{1}{r(1-r)} K_{r}^{2} = 0, \qquad \forall \, r \in (0,1). 
\end{equation*}
Since we can choose $k \in C^{2}_{c}(0,1)$ such that $K = Q k$ does not satisfy the above equation, we obtain a contradiction.  
\end{proof}
 
\section{Conjectures and open problems}
\label{sect_conj_dynamics}

Theorem \ref{statement_ibpf} above enables us to conjecture the structure of the Bessel SPDEs for $\delta < 3$. The idea is that the right-hand side of the IbPFs \eqref{exp_fst_part_ibpf_a_b} (respectively \eqref{exp_fst_part_ibpf_a_b_1}) corresponds to the logarithmic derivative of the measure $P^{\delta}$ for $\delta \in (0,3) \setminus \{1\}$ (resp. $\delta=1$), which should yield the drift in the SPDEs we are looking for. 
More precisely, considering for instance the case $\delta \in (1,3)$, for all $\Phi \in \mathcal{S}$, we may rewrite the last term in the IbPF \eqref{exp_fst_part_ibpf_a_b} as follows
\[
\begin{split}
&- \kappa(\delta) \int_{0}^{1} h_{r} \int_{0}^{\infty} a^{\delta-4}\left( \Sigma^\delta_r\left(\Phi(X) \,|\, a\right) - \Sigma^\delta_r\left(\Phi(X) \,|\, 0\right) \right) \d a \, \d r = \\
&= - \kappa(\delta) \, \lim_{\epsilon \to 0} \lim_{\eta \to 0} \mathbb{E} \left[ \Phi(X)  \int_{0}^{1} h_{r} \, \left( \frac{\mathbf{1}_{\{X_r \geq \epsilon\}}}{X_r^{3}} - 2 \frac{\epsilon^{\delta-3}}{3-\delta} \frac{\rho_{\eta}(X_r)}{X_r^{\delta-1}} \right) \d r \right], 
\end{split}
\]
where the mollifying functions $\rho_\eta, \eta >0$ are as in \eqref{def_mollifier}.  
As a consequence of this equality, we may write formally the gradient dynamics corresponding to $P^{\delta}$, $\delta \in (1,3)$, as follows
\[  \partial_{t} u = \frac{1}{2} \partial^{2}_{x} u + \xi + \frac{\kappa(\delta)}{2} \lim_{\epsilon \to 0} \lim_{\eta \to 0} \, \left( \frac{\mathbf{1}_{\{u \geq \epsilon\}}}{u^{3}} - 2 \frac{\epsilon^{\delta-3}}{3-\delta}  \frac{\rho_{\eta}(u)}{u^{\delta-1}} \right), \]
where $\xi$ denotes space-time white noise on $\mathbb{R}_{+} \times (0,1)$. Assuming now the existence of a local time process $(\ell^a_{t,x})_{x \in (0,1), t, a \geq 0}$ satisfying the occupation times formula \eqref{otf}
and possessing sufficient regularity at $a=0$, we could in turn write 
\[\lim_{\epsilon \to 0} \lim_{\eta \to 0} \, \left( \frac{\mathbf{1}_{u \geq \epsilon}}{u^{3}} - 2\frac{\epsilon^{\delta-3}}{3-\delta} \frac{\rho_{\eta}(u)}{u^{\delta-1}} \right) = \int_{0}^{+\infty} \ a^{\delta-4} (\ell^a_{t,x} - \ell^0_{t,x}) \d a, \]
so the SPDE could be written:
\[  \partial_{t} u = \frac{1}{2} \partial^{2}_{x} u + \xi + \frac{\kappa(\delta)}{2} \frac{\partial}{\partial  t} \int_{0}^{+\infty} \ a^{\delta-4} (\ell^a_{t,x} - \ell^0_{t,x}) \d a. \]
The same reasoning can be done for $\delta \in (0,1)$, yielding for that case
\[  \partial_{t} u = \frac{1}{2} \partial^{2}_{x} u + \xi + \frac{\kappa(\delta)}{2}  \frac{\partial}{\partial  t} \int_{0}^{+\infty} \ a^{\delta-4} \,\mathcal{T}^{\,2}_{a} \ell^{(\cdot)}_{t,x} \d a. \]
As for the critical case $\delta = 1$, as shown in Section \ref{sect_Dirichlet}, the dynamics is formally given by \eqref{formal1}, which we can rewrite using the local times as follows:
\[
 \partial_{t} u = \frac{1}{2} \partial^{2}_{x} u + \xi -  \frac{1}{8} \frac{\partial}{\partial t} \frac{\partial^2}{\partial a^{2}} \ell^{a}_{t,x} \biggr\rvert_{a=0}.
\] 
In all the SPDEs above, the unknown would be the couple $(u,\ell)$, where $u$ is a continuous nonnegative function on $\mathbb{R}_{+} \times (0,1)$, and, for all $x \in (0,1)$, $(\ell^{a}_{t}(x))_{a, t \geq 0}$ is a family of occupation times satisfying \eqref{otfo}.

These conjectures raise several problems. Indeed, assuming that the process $u$ can be constructed - as done above for the case $\delta=1$ - it is at present unknown whether a family of occupation times $\ell$ satisfying \eqref{otfo} should exist and, if it does, whether it has the requested differentiability property. Moreover, pathwise uniqueness for such equations is at present an open problem. 
For instance, due to the lack of monotonicity, the techniques used in \cite{nualart1992white} to define a solution to the stochastic heat equation with reflection would not be of any help. We stress that the analogous SDE case of Bessel processes of dimension $\delta \in (0,1)$ is also a problem of interest in itself; these processes are not semi-martingales, but nonetheless satisfy the stochastic equation
\[ X_{t} = x + \frac{\delta-1}{2} \int_{0}^{+\infty} a^{\delta-2} (\ell^{a}_{t} - \ell^{0}_{t}) \d a + B_{t}, \]
where $\left(\ell^{a}_{t}\right)_{a,t \geq 0}$ is the diffusion local times process of the Bessel process $(X_{t})_{t \geq 0}$ (see \cite{revuz2013continuous}, Chapter XI, ex. 1.26). Even in this one-dimensional context, the only known method for solving this equation is to consider $Y_t:=X_t^2$ and show pathwise uniqueness for $Y$; this method breaks down for SPDEs since the It\^o formula produces very complicated terms, see the discussion in the Introduction.

The Dirichlet form techniques used in Section \ref{sect_Dirichlet} above to construct $u$ in the case $\delta=1$ can also be applied successfully to treat the case $\delta=2$, see the forthcoming paper \cite{henri2018bessel}. However, for $\delta\in\,]0,3[\,\setminus\{1,2\}$, it is not even known whether the form which naturally generalizes $(\mathcal{E},\mathcal{F} \mathcal{C}^{\infty}_{b}(K))$ in Proposition \ref{closability} is closable and whether its closure is a quasi-regular Dirichlet form . 

We recall the main result of \cite{dalang2006hitting}: for all $\delta\geq 3$, we set
\[
\zeta(\delta):=\sup\{k\geq\N: \exists t>0, \, 0<x_1<\ldots<x_k<1, \, u(t,x_i)=0 \quad i=1,\ldots,k\},
\]
where $u$ is the solution to the $\delta$-Bessel SPDE \eqref{spde>3}-\eqref{spde=3}. Then we have
\begin{equation}\label{conj}
{\mathbb P}\left( \zeta(\delta)>\frac 4{\delta-2} \right)=0.
\end{equation}
In other words, a.s. $u$ hits the obstacle $0$ in at most $\lfloor\frac 4{\delta-2}\rfloor$ space points simultaneously in time. It is very tempting to conjecture that \eqref{conj} holds for all $\delta>2$
in other words, the $\delta$-Bessel SPDE would hit 0 at finitely many space points simultaneously in time for any $\delta>2$, but the number of such hitting points would tend to $+\infty$ as $\delta\downarrow 2$. The fact that $\delta=2$ is the critical value for this behaviour is clearly
related to the fact that $\delta=2$ is also the critical dimension for the probability that the
$\delta$-Bessel process or bridge hit 0.

The transition between $\delta\geq 3$ and $\delta<3$ is visible at the level of the invariant measure, namely the $\delta$-Bessel bridge, since in the former case the measure is log-concave, 
while this property is lost in the latter case. Therefore the techniques of \cite{ASZ} based on 
optimal transport and gradient flows in metric spaces fail for $\delta<3$. In the same vein, the Strong Feller property holds easily
for $\delta\geq 3$, while it is an open problem for $\delta<3$, again because the drift of the SPDE
becomes highly non-dissipative. Still, the recent paper \cite{Henri18} of the first author shows that
Bessel processes of dimension $\delta<1$ are Strong Feller even if their drift contains a renormalised local time. Moreover Tsatsoulis and Weber \cite{TW} have proved that the 2-dimensional stochastic quantization equation satisfies a Strong Feller property, although it is an equation which
needs renormalisation; also Hairer and Mattingly \cite{HM18} have proved the Strong Feller property for a large class of
equations with renormalised drifts. All this suggests that there may be hope that this technically very useful
property holds also for $\delta$-Bessel SPDEs with $\delta<3$.

\appendix

\section{Proofs of two technical results}\label{Proofs}
\begin{proof}[Proof of Proposition \ref{closability}]
Since $D(\Lambda)$ contains all globally Lipschitz functions on $H$, for all $f \in \mathcal{F} \mathcal{C}^{\infty}_{b}(K)$ we have $f \circ j \in D(\Lambda)$. 
A simple calculation shows that for any $f\in \mathcal{F} \mathcal{C}^{\infty}_{b}(K)$ of the form \eqref{Fexp} we have
\begin{equation}
\label{derivative_functional_abs_val}
\nabla (f\circ j)(z) = \nabla f (j(z)) \, \text{sgn}(z).
\end{equation}  
Hence, for all $f,g \in \mathcal{F} \mathcal{C}^{\infty}_{b}(K)$, we have 
\[ \begin{split}
\mathcal{E}(f,g)  &= \frac{1}{2} \int \langle \nabla f(x) , \nabla g(x) \rangle \d \nu(x) = \frac{1}{2} \int \langle \nabla f(j(z)) , \nabla g(j(z)) \rangle \d \mu (z) \\
&= \frac{1}{2} \int \langle \nabla (f \circ j)(z) , \nabla (g \circ j)(z) \rangle \d \mu (z) = \Lambda(f \circ j, g \circ j),
\end{split} \]
where the third equality follows from \eqref{derivative_functional_abs_val}. This shows that the bilinear symmetric form $(\mathcal{E},\mathcal{F} \mathcal{C}^{\infty}_{b}(K))$ admits as an extension the image of the Dirichlet form $(\Lambda, D(\Lambda))$ under the map $j$. Since $\mathcal{F} \mathcal{C}^{\infty}_{b}(K)$ is dense in $L^{2}(\nu)$, this extension is a Dirichlet form. In particular, $(\mathcal{E},\mathcal{F} \mathcal{C}^{\infty}_{b}(K))$ is closable, its closure $(\mathcal{E},D (\mathcal{E}))$ is a Dirichlet form, and we have the isometry property \eqref{isometry}.

There remains to prove that the Dirichlet form $(\mathcal{E},D (\mathcal{E}))$ is quasi-regular. Since it is the closure of $(\mathcal{E},\mathcal{F} \mathcal{C}^{\infty}_{b}(K))$, it suffices to show that the associated capacity is tight. Since $K$ is separable, we can find a countable dense subset $\{ y_{k}, \, k \in \mathbb{N} \} \subset K$ such that $y_k \neq 0$ for all $k \in \mathbb{N}$. 

Let now $\varphi \in C^{\infty}_{b}(\mathbb{R})$ be an increasing function such that $\varphi(t)=t$ for all $t \in [-1,1]$ and $\|\varphi'\|_{\infty} \leq 1$. For all $m \in \mathbb{N}$, we define the function $v_{m} : K \to \mathbb{R}$ by
\[ v_{m}(z) := \varphi(\|z-y_{m}\|), \quad z \in K.\]
Moreover, we set, for all $n \in \mathbb{N}$
\[w_{n}(z) := \underset{m \leq n}{\inf} v_{m}(z), \quad z \in K.\] 
We claim that $w_{n} \in D(\mathcal{E})$, $n \in \mathbb{N}$, and that $w_{n} \underset{n \to \infty}{\longrightarrow} 0$, $\mathcal{E}$ quasi-uniformly in $K$. Assuming this claim for the moment, for all $k \geq 1$ we can find a closed subset $F_{k}$ of $K$ such that $\text{Cap} (K \setminus F_{k}) < 1/k$, and $w_{n} \underset{n \to \infty}{\longrightarrow} 0$ uniformly on $F_{k}$. Hence, for all $\epsilon >0$, we can find $n \in \mathbb{N}$ such that $w_{n} < \epsilon$ on $F_{k}$. Therefore
\[ F_{k} \subset \underset{m \leq n}{\bigcup} B(y_{m}, \epsilon) \]
where $B(y, r)$ is the open ball in $K$ centered at $y \in K$ with radius $r >0$.
This shows that $F_{k}$ is totally bounded. Since it is, moreover, complete as a closed subspace of a complete metric space, it is compact, and the tightness of $\text{Cap}$ follows.

We now justify our claim. For all $i \in \mathbb{N}$, we set $l_i := \|y_i\| ^{-1} \, y_i$. Then for all $i \geq 1$, $l_{i} \in K$, $\|l_{i}\| = 1$ and, for all $z \in K$
\[ \|z\| = \underset{i \geq 0}{\sup} \, \langle l_{i}, z \rangle. \]
Let $m \in \mathbb{N}$ be fixed. For all $i \geq 0$, let
$u_{i}(z) := \underset{j \leq i}{\sup} \, \, \varphi( \, \langle l_{j}, z- y_{m} \rangle \, )$, $z \in K$.  
We have $u_{i} \in D(\mathcal{E})$, and, for $\nu$ - a.e. $z \in K$
\[\sum_{k=1}^{\infty} \frac{\partial u_{i}}{\partial e_{k}} (z) ^{2} \leq \underset{j \leq i}{\sup} \left(  \sum_{k=1}^{\infty} \varphi'(\langle l_{j}, z - y_{m} \rangle )^{2} \, \langle l_{j}, e_{k} \rangle ^{2} \right) \leq 1,  \]
whence $\mathcal{E}(u_{i}, u_{i})\leq 1$.
By the definition of $v_{m}$, as $i \to \infty$, $u_{i} \uparrow v_{m}$ on $K$, hence in $L^{2}(K, \nu)$. By  \cite[I.2.12]{ma2012introduction}, we deduce that $v_{m} \in D(\mathcal{E})$, and that
$
\mathcal{E}(v_{m}, v_{m}) \leq 1.
$    
Therefore, for all $n \in \mathbb{N}$, $w_{n} \in D(\mathcal{E})$, and
$ \mathcal{E}(w_{n}, w_{n}) \leq 1. 
$ 
But, since $\{ y_{k}, \, k \in \mathbb{N} \}$ is dense in $K$, as $n \to \infty$, $w_{n} \downarrow 0$ on $K$. Hence $w_{n} \underset{n \to \infty}{\longrightarrow} 0$ in $L^{2}(K, \nu)$. This and the previous bound imply, by \cite[I.2.12]{ma2012introduction}, that the Ces\`{a}ro means of some subsequence of $(w_{n})_{n \geq 0}$ converge  to $0$ in $D(\mathcal{E})$. By \cite[III.3.5]{ma2012introduction}, some subsequence thereof converges $\mathcal{E}$ quasi-uniformly to $0$. But, since $(w_{n})_{n \geq 0}$ is non-increasing, we deduce that it converges $\mathcal{E}$-quasi-uniformly to $0$. The claimed quasi-regularity follows. There finally remains to check that $(\mathcal{E}, D(\mathcal{E}))$ is local in the sense of Definition \cite[V.1.1]{ma2012introduction}. Let $u,v \in D(\mathcal{E})$ satisfying $\text{supp}(u) \cap \text{supp}(v) = \emptyset$. Then, $u \circ j$ and $v \circ j$ are two elements of $D(\Lambda)=W^{1,2}(\mu)$ with disjoint supports, and, recalling \eqref{isometry}, we have 
\[ \mathcal{E}(u,v) = \Lambda(u \circ j,v \circ j) = \frac{1}{2} \int_{H} \nabla (u \circ j) \cdot \nabla (v \circ j) \d \mu=0. \]
The claim follows.       
\end{proof}

\begin{proof}[Proof of Lemma \ref{density}]
Recall that $D(\mathcal{E})$ is the closure under the bilinear form $\mathcal{E}_{1}$
of the space $\mathcal{F} \mathcal{C}^{\infty}_{b}(K)$ of functionals of the form $F = \Phi \bigr \rvert_{K}$, where $\Phi \in \mathcal{F} \mathcal{C}^{\infty}_{b}(H)$. 
Therefore, to prove the claim, it suffices to show that for any functional $\Phi \in \mathcal{F} \mathcal{C}^{\infty}_{b}(H)$ and all $\epsilon>0$, there exists $\Psi \in \mathscr{S}$ such that $\mathcal{E}_1(\Phi-\Psi,\Phi-\Psi) < \epsilon $.

Let $\Phi \in \mathcal{F} \mathcal{C}^{\infty}_{b}(H)$. We set for all $\epsilon > 0$
\[ \Phi_{\epsilon}(\zeta) := \Phi(\sqrt{\zeta^{2} + \epsilon}), \quad \zeta \in H.  \] 
A simple calculation shows that $\Phi_{\epsilon} \underset{\epsilon \to 0}{\longrightarrow} \Phi$ 
and $\nabla \Phi_{\epsilon} \underset{\epsilon \to 0}{\longrightarrow} \nabla \Phi$ pointwise, with
uniform bounds $\|\Phi_{\epsilon}\|_{\infty} \leq \| \Phi \|_{\infty}$ and $ \| \nabla \Phi_{\epsilon} \|_{\infty} \leq  \| \nabla \Phi \|_{\infty}$. Hence, by dominated convergence, $\mathcal{E}_1 (\Phi_{\epsilon} - \Phi, \Phi_{\epsilon} - \Phi) \underset{\epsilon \to 0}{\longrightarrow} 0$. Then, introducing for all $d \geq 1$ $(\zeta^{d}_{i})_{1 \leq i \leq d}$ the orthonormal family in $L^{2}(0,1)$ given by
\[
 \zeta^{d}_{i} := \sqrt{d} \ \mathbf{1}_{[\frac{i-1}{d}, \frac{i}{d}[}, \quad i = 1, \ldots, d, 
\]
and setting
\[ \Phi^{d}_{\epsilon}(\zeta) :=  \Phi_{\epsilon} \left( \left( \sum_{i=1}^{d} \langle \zeta_{d,i}, \zeta^{2} \rangle  \right)^{\frac12} \right) = \Phi \left( \left( \sum_{i=1}^{d} \langle \zeta_{d,i}, \zeta^{2} \rangle   + \epsilon \right)^{\frac12} \right) , \quad \zeta \in H,  \]
again we obtain the convergence $\mathcal{E}_1(\Phi^{d}_{\epsilon} - \Phi_{\epsilon}, \Phi^{d}_{\epsilon} - \Phi_{\epsilon}) \underset{d \to \infty}{\longrightarrow} 0$.

There remains to show that any fixed functional of the form 
\[ \Phi (\zeta) = f\left( \langle \zeta_{1}, \zeta^{2} \rangle, \ldots, \langle \zeta_{d}, \zeta^{2} \rangle \right), \quad \zeta \in H \]
with $d \geq 1$, $f \in C^{1}_{b}(\mathbb{R}_{+}^{d})$, and $(\zeta_{i})_{i=1, \ldots, d}$  
a family of elements of $K$, 
can be approximated by elements of $\mathscr{S}$. Again by dominated convergence, we can suppose
that $f$ has compact support in $\mathbb{R}_{+}^{d}$. We define $g\in C^{1}_{b}([0,1]^{d})$,
\[ g(y) := f(-\ln(y_{1}), \cdots, -\ln(y_{d})), \quad y \in \,]0,1]^{d}, \] 
and $g(y):=0$ if $y_i=0$ for any $i=1,\ldots,d$. 
By a differentiable version of the Weierstrass Approximation Theorem (see Theorem 1.1.2 in \cite{llavona1986approximation}), there exists a sequence $(p_{k})_{k \geq 1}$ of polynomial functions converging to $g$ for the $C^{1}$ topology on $[0,1]^{d}$. Defining for all $k \geq 1$ the function $f_{k}: \mathbb{R}_{+}^{d} \to \mathbb{R}$ by
\[ f_{k}(x) = p_{k}(e^{-x_{1}}, \cdots, e^{-x_{d}}), \quad x \in \mathbb{R}_{+}^{d}, \]  
we define $\Phi_{k} \in \mathscr{S}$ by
\[ \Phi_{k} (\zeta) = f_{k} \left( \langle \zeta_{1}, \zeta^{2} \rangle, \ldots, \langle \zeta_{d}, \zeta^{2} \rangle \right), \quad \zeta \in H. \]
Since $p_{k} \underset{k \to \infty}{\longrightarrow} g$ for the $C^{1}$ topology on $[0,1]^{d}$, $f_{k} \underset{k \to \infty}{\longrightarrow} f$ uniformly on $\mathbb{R}_{+}^{d}$ together with its first order derivatives. Hence, it follows that $\Phi_{k} \underset{k \to \infty}{\longrightarrow} \Phi$ pointwise on $K$ together with its
gradient. It also follows that there is some $C>0$ such that for all $k \geq 1$
\[ \forall  \zeta \in K, \quad |\Phi_{k}(\zeta)|^{2} + \|\nabla \Phi_{k}(\zeta)\|^{2} \leq C(1+ \|\zeta \|^{2}). \]
Since the quantity in the right-hand side is $\nu$ integrable in $\zeta$, it follows by dominated convergence that $\mathcal{E}_1(\Phi_{k}-\Phi, \Phi_{k}-\Phi) \underset{k \to \infty}{\longrightarrow} 0$. This yields the claim.  
\end{proof}

\bibliographystyle{amsplain}
\bibliography{IBPF_Bessel_Revised}

\end{document}